\documentclass[12pt,twoside]{amsart}

\usepackage[latin1]{inputenc}

\usepackage{latexsym}
\usepackage{exscale}

\usepackage{amsmath,amsfonts,amssymb,amsthm,bbm,fullpage}

\usepackage{fancyhdr,graphicx,color}

\numberwithin{equation}{section}

\newtheorem{thm}{Theorem}[section]

\newtheorem{lem}[thm]{Lemma}

\newtheorem{cor}[thm]{Corollary}

\newtheorem{prop}[thm]{Proposition}

\newtheorem{rem}[thm]{Remark}

\newtheorem{dfn}[thm]{Definition}
\newcommand{\diam}{diam}
\newcommand{\aver}[1]{-\hskip-0.46cm\int_{#1}}

\newcommand{\ind}{1\hspace{-2.5 mm}{1}}

\DeclareMathOperator{\supp}{supp}

\begin{document}
\allowdisplaybreaks
\title{Real interpolation of Sobolev spaces associated to a weight}
\author{Nadine BADR}
\address{N.Badr
\\
Universit\'e de Paris-Sud, UMR du CNRS 8628
\\
91405 Orsay Cedex, France} \email{nadine.badr@math.u-psud.fr}

%\date{\today}

\begin{abstract}
 We study the interpolation property of Sobolev spaces of order 1 denoted by $W^{1}_{p,V}$, arising from Schr\"{o}dinger operators with positive potential. We show that for $1\leq p_1<p<p_2<q_{0}$ with $p>s
 _0$, $W^{1}_{p,V}$ is a real interpolation space between $W_{p_1,V}^{1}$ and $W_{p_2,V}^{1}$ on some classes of manifolds and Lie groups. The constants $s_{0},\,q_{0}$ depend on our hypotheses. 
\end{abstract}
\maketitle
\tableofcontents
\section{Introduction}
In \cite{auscher3}, the Schr\"{o}dinger operator $-\Delta+V$ on $\mathbb{R}^{n}$ with $V\in A_{\infty}$, the Muckenhoupt class (see \cite{garcia}), is studied and the question whether the spaces defined by the norm $\|f\|_{p}+\|\,|\nabla f|\,\|_{p}+\|V^{\frac{1}{2}}f\|_{p}$ or ($\|\,|\nabla f|\,\|_{p}+\|V^{\frac{1}{2}}f\|_{p}$) interpolate is posed. In fact, it is shown that:
$$
\|\,|\nabla f|\,\|_{p}+\|V^{\frac{1}{2}}f\|_{p}\sim \|(-\Delta+V)^{\frac{1}{2}}f\|_{p}
$$
whenever $1<p<\infty$ and $p\leq 2q$, $f\in C_{0}^{\infty}(\mathbb{R}^{n})$, where $q>1$ is a Reverse H\"{o}lder exponent of $V$. Hence the question of interpolation can be solved a posteriori using functional calculus and interpolation of $L_{p}$ spaces. However, it is reasonable to expect a direct proof. 

Here we provide such an argument with $p$ lying in an interval depending on the Reverse H\"{o}lder exponent of $V$ by estimating the $K$-functional of real interpolation. The particular case $V=1$ is treated in \cite{badr1} (also $V=0$). The method is actually valid on some Lie groups and even some Riemannian manifolds in which we place ourselves.

Let us come to statements:
\begin{dfn}
Let $M$ be a Riemannian manifold, $V\in A_{\infty}$. Consider  for $1\leq p<\infty$, the vector space $E_{p,V}^{1}$ of $C^{\infty}$ functions $f$ on $M$ such that $f,\,|\nabla f|$ and $ Vf \in L_{p}(M)$. We define the Sobolev space $W_{p,V}^{1}(M)=W_{p,V}^{1}$ as the completion of $E_{p,V}^{1}$ for the norm 
$$
\|f\|_{W_{p,V}^{1}}= \|f\|_{p}+\|\,|\nabla f|\,\|_{p}+ \|Vf\|_{p}.
$$
\end{dfn}
\begin{dfn} We denote by $W_{\infty,V}^{1}(M)=W_{\infty,V}^{1}$ the space of all bounded Lipschitz functions $f$ on $M$ with $\|Vf\|_{\infty}<\infty$.
\end{dfn}  
We have the following interpolation theorem for the non-homogeneous Sobolev spaces $W_{p,V}^{1}$:
\begin{thm}\label{IS}
Let $M$ be a complete Riemannian manifold satisfying  a local doubling property $(D_{loc})$. Let $V\in RH_{qloc}$ for some $1< q\leq\infty$. Assume that $M$ admits a local Poincar\'{e} inequality $(P_{sloc})$ for some $1\leq s<q$. Then for $1\leq r\leq s <p <q$, $W_{p,V}^{1}$ is a real interpolation space between $W_{r,V}^{1}$ and  $W_{q,V}^{1}$.
\end{thm}
\begin{dfn}
Let $M$ be a Riemannian manifold, $V\in A_{\infty}$. Consider for $1\leq p<\infty$, the vector space $\dot{W}_{p,V}^{1}$ of 
 distributions $ f $ such that $|\nabla f|$ and $ Vf \in L_{p}(M)$. 
 It is well known that the elements of $\dot{W}_{p,V}^{1}$ are in  $L_{p,loc} $. We equip $\dot{W}_{p,V}^{1}$  with the semi norm 
$$
\|f\|_{\dot{W}_{p,V}^{1}}=\|\,|\nabla f|\,\|_{p}+ \|Vf\|_{p}.
$$
\end{dfn}
In fact, this expression is a norm since $V\in A_{\infty}$ yields $V>0 \; \mu-a.e.$.
\begin{dfn} We denote $\dot{W}_{\infty,V}^{1}(M)=\dot{W}_{\infty,V}^{1}$ the space of all Lipschitz functions $f$ on $M$ with $\|Vf\|_{\infty}<\infty$.
\end{dfn}
For the homogeneous Sobolev spaces $\dot{W}_{p,V}^{1}$, we have
 \begin{thm}\label{IHS}  Let $M$ be a complete Riemannian manifold satisfying $(D)$. Let $V\in RH_{q}$  for some $1< q\leq\infty$ and assume that $M$ admits a Poincar\'{e} inequality $(P_{s})$ for some $1\leq s<q$. Then, for $1\leq r\leq s< p<q$, $\dot{W}_{p,V}^{1}$ is a real interpolation space between $\dot{W}_{r,V}^{1}$ and $\dot{W}_{q,V}^{1}$. 
\end{thm}
It is known that if $V\in RH_{q}$ then $V+1\in RH_{q}$ with comparable constants. Hence part of Theorem \ref{IS} can be seen as a corollary of Theorem \ref{IHS}. But the fact that $V+1$ is bounded away from $0$ also allows local assumptions in Theorem \ref{IS}, which is why we distinguish in this way the non-homogeneous and the homogeneous case.

The proof of Theorem \ref{IS} and Theorem \ref{IHS} is done by estimating the $K$-functional of interpolation. We were not able to obtain a characterization of the $K$-functional. However, this suffices for our needs. When $q=\infty$ (for example if $V$ is a positive polynomial on $\mathbb{R}^{n}$) and $r=s$, then there is a characterization.
The key tools to estimate the $K$-functional  will be a Calder\'{o}n-Zygmund decomposition for Sobolev functions and the Fefferman-Phong inequality (see section 3).
 
We end this introduction with a plan of the paper. In section 2, we review the notions of doubling property, Poincar\'{e} inequality, Reverse H\"{o}lder classes as well as the real $K$ interpolation method. At the end of this section, we summarize some properties for the Sobolev spaces defined above under some additional hypotheses on $M$ and $V$. Section 3 is devoted to give the main tools: the Fefferman-Phong inequality and a Calder\'{o}n-Zygmund decomposition adapted to our Sobolev spaces. In section 4, we estimate the $K$-functional of real interpolation for non-homogeneous Sobolev spaces in two steps: first of all for the global case and secondly for the local case. We interpolate and get Theorem \ref{IS} in section 5. Section 6 concerns the proof of Theorem \ref{IHS}. Finally, in section 7, we apply our interpolation result to the case of Lie groups with an appropriate definition of $W_{p,V}^{1}$. 
\\
\\
\textit{Acknowledgements.} I thank my Ph.D advisor P. Auscher for the useful discussions on the topic of this paper.
\
\section{Preliminaries}
Throughout this paper we write $\ind_{E}$  for the characteristic function of a set $E$ and $E^{c}$  for the complement of $E$. For a ball $B$ in a metric space, $\lambda B$  denotes the ball co-centered with $B$ and with radius $\lambda$ times that of $B$. Finally, $C$ will be a constant
that may change from an inequality to another and we will use $u\sim
v$ to say that there exist two constants $C_{1}$, $C_{2}>0$ such that $C_{1}u\leq v\leq
C_{2}u$. Let $M$ denotes a complete non-compact Riemannian manifold. We write $\mu$ for the Riemannian measure on $M$, $\nabla$ for the Riemannian gradient, $|\cdot|$ for the length on the tangent space (forgetting the subscript $x$ for simplicity) and $\|\cdot\|_{p}$ for the norm on $ L_{p}(M,\mu)$, $1 \leq p\leq +\infty.$
\subsection{The doubling property and Poincar\'{e} inequality}
\begin{dfn} Let $(M,d,\mu)$ be a Riemannian manifold. Denote by $B(x, r)$ the open ball of center $x\in M $ and radius $r>0$. One says that $M$ satisfies the  local doubling property $(D_{loc})$ if there exist constants $r_{0}>0$, $0<C=C(r_{0})<\infty$, such that for all $x\in M,\, 0<r< r_{0} $ we have
\begin{equation*}\tag{$D_ {loc}$}
\mu(B(x,2r))\leq C \mu(B(x,r)).
\end{equation*}
Furthermore, $M$ satisfies a global doubling property or simply doubling property $(D)$ if one can take $r_{0}=\infty$.
We also say that $\mu$ is a locally (resp. globally) doubling Borel measure.
\end{dfn}
\noindent Observe that if $M$ satisfies $(D)$ then
$$
 \diam(M)<\infty\Leftrightarrow\,\mu(M)<\infty\; \textrm{(\cite{ambrosio1})}. 
$$
 %\begin{lem}\label{D} Let $M$ be a Riemannian manifold satisfying $(D)$ with the constant $C$ and let $s=log_{2}C$. Then for $x_{0}\in M$, $r_{0}>0$, we have
 %$$
 %\frac{\mu(B(x,r))}{\mu(B(x_{0},r_{0}))}\geq 4^{-s} (\frac{r}{r_{0}})^{s}
 %$$
 %whenever $x\in B(x_{0},r_{0})$ and $r\leq r_{0}$.
 %\end{lem}
\begin{thm}[Maximal theorem]\label{MIT}(\cite{coifman2})
Let $M$ be a Riemannian manifold satisfying $(D)$. Denote by $\mathcal{M}$ the uncentered Hardy-Littlewood maximal function over open balls of $M$ defined by
 $$
 \mathcal{M}f(x)=\underset{B:x\in B}{\sup}|f|_{B}
 $$ 
 where $ \displaystyle f_{E}:=\aver{E}f d\mu:=\frac{1}{\mu(E)}\int_{E}f d\mu.$
Then

\begin{itemize}
\item[1.] $\mu(\left\lbrace x:\,\mathcal{M}f(x)>\lambda\right\rbrace)\leq \frac{C}{\lambda}\int_{X}|f| d\mu$ for every $\lambda>0$;
\item[2.] $\|\mathcal{M}f\|_{p}\leq C_{p} \|f\|_{p}$, for $1<p\leq\infty$.
\end{itemize}
\end{thm}
\subsection{Poincar\'{e} inequality}
\begin{dfn}[Poincar\'{e} inequality on $M$] Let $M$ be a complete Riemannian manifold, $1\leq s<\infty$. We say that $M$ admits \textbf{a local Poincar\'{e} inequality $(P_{sloc})$} if there exist constants $r_{1}>0,\,C=C(r_{1})>0$ such that, for every function $f\in C^{\infty}_{0}$, and every ball $B$ of $M$ of radius $0<r<r_{1}$, we have
\begin{equation*}\tag{$P_{sloc}$}
\aver{B}|f-f_{B}|^{s} d\mu \leq C r^{s} \aver{B}|\nabla f|^{s}d\mu.
\end{equation*}
$M$ admits a global Poincar\'{e} inequality $(P_{s})$ if we can take $r_{1}=\infty$ in this definition.\\
\end{dfn}
\begin{rem} By density of $C_{0}^{\infty}$ in $W_{s}^{1}$, if $(P_{sloc})$ holds for every function $f\in C_{0}^{\infty}$, then it holds for every $f\in W_{s}^{1}$.
\end{rem}
 Let us recall some known facts about Poincar\'{e} inequality with varying $q$. It is known that $(P_{qloc})$ implies $(P_{ploc})$ when $p\geq q$ (see \cite{hajlasz4}). Thus, if the set of  $q$ such that $(P_{qloc})$ holds is not empty, then it is an interval unbounded on the right. A recent result from Keith-Zhong \cite{keith3} asserts that this interval is open in $[1,+\infty[$ in the following sense:
\begin{thm}\label{kz} Let $(X,d,\mu)$ be a complete metric-measure space with $\mu$ locally doubling
and admitting a local Poincar\'{e} inequality $(P_{qloc})$, for  some $1< q<\infty$.
Then there exists $\epsilon >0$ such that $(X,d,\mu)$ admits
$(P_{ploc})$ for every $p>q-\epsilon$ (see \cite{keith3} and section 4 in \cite{badr1}). 
\end{thm}
\subsection{Reverse H\"{o}lder classes}
\begin{dfn} Let $M$ be a Riemannian manifold. A weight $w$ is a non-negative locally integrable function on $M$. The reverse H\"{o}lder classes are defined in the following way: $w\in RH_{q},\,1<q<\infty$, if there exists a constant $C$ such that for every ball $B\subset M$
\begin{equation}\label{rhq}
\left(\aver{B}w^{q}d\mu\right)^{\frac{1}{q}}\leq C\aver{B}wd\mu.
\end{equation}
The endpoint $q=\infty$ is given by the condition: $w\in RH_{\infty}$ whenever, for any ball $B$,
\begin{equation}\label {rhi}
w(x)\leq C\aver{B}w \quad \textrm{for } \mu-a.e. \;x\in B.
\end{equation}
We say that $w\in RH_{qloc}$ for some $1<q< \infty$ (resp. $q=\infty$) if  there exists $r_{2}>0$ such that (\ref{rhq}) (resp. (\ref{rhi})) holds for all balls $B$ of radius $0<r<r_{2}$.
\\
The smallest  $C$ is called the $RH_{q}$ (resp. $RH_{qloc}$) constant of $w$.
\end{dfn}
\begin{prop}\label{CRO}
\begin{itemize}
\item[1.] $RH_{\infty}\subset RH_{q}\subset RH_{p}$ for $1<p\leq q\leq \infty$.
\item[2.] If $w\in RH_{q}$, $1<q<\infty$, then there exists $q<p<\infty$ such that $w\in RH_{p}$.
%\item[3.]  $w\in RH_{q},\,1<q <\infty$ if and only if  $\displaystyle w^{q}\in A_{\infty}=\displaystyle\cup_{1<q\leq \infty}RH_{q}$.
\item[3.] $A_{\infty}=\bigcup_{1<q\leq \infty} RH_{q}$. 
\end{itemize}
\end{prop} 
\begin{proof}
These properties are standard, see for instance \cite{garcia}.  
\end{proof}
\begin{prop}\label{CR}(see section 11 in \cite{auscher3}, \cite{johnson})
Let $V$ be a non-negative measurable function. Then the following properties are equivalent:
\begin{itemize}
\item[1.] $V\in A_{\infty}$.
\item[2.] For all $r\in]0,1[,\,V^{r}\in RH_{\frac{1}{r}}$.
\item[3.] There exists $r\in ]0,1[,\,V^{r}\in RH_{\frac{1}{r}}$.
\end{itemize}
\end{prop}
\begin{rem}
Propositions \ref{CRO} and \ref{CR} still hold in the local case, that is, when the weights are considered in a local reverse H\"{o}lder class $RH_{qloc}$ for some $1<q\leq \infty$.
\end{rem} 
\subsection{The $K$ method of real interpolation} The reader is referred to \cite{bennett}, \cite{bergh} for details on the development of this theory. Here we only recall the essentials to be used in the sequel. 

Let $A_{0}$, $A_{1}$ be  two normed vector spaces embedded in a topological Hausdorff vector space $V$, and define for $a\in A_{0}+A_{1}$ and $t>0$,
$$
K(a,t,A_{0},A_{1})=\displaystyle \inf_{a
=a_{0}+a_{1}}(\| a_{0}\|_{A_{0}}+t\|
a_{1}\|_{A_{1}}).
$$

For $0<\theta< 1$, $1\leq q\leq \infty$, we denote by $(A_{0},A_{1})_{\theta,q}$ the interpolation space between $A_{0}$ and $A_{1}$:
\begin{displaymath}
	(A_{0},A_{1})_{\theta,q}=\left\lbrace a \in A_{0}+A_{1}:\|a\|_{\theta,q}=\left(\int_{0}^{\infty}(t^{-\theta}K(a,t,A_{0},A_{1}))^{q}\,\frac{dt}{t}\right)^{\frac{1}{q}}<\infty\right\rbrace.
\end{displaymath}
It is an exact interpolation space of exponent $\theta$ between $A_{0}$ and $A_{1}$, see \cite{bergh} Chapter II.
\begin{dfn}
Let $f$  be a measurable function on a measure space $(X,\mu)$. We denote by $f^{*}$ its decreasing rearrangement function: for every $t>0$,
$$
f^{*}(t)=\inf \left\lbrace\lambda :\, \mu (\left\lbrace x:\,|f(x)|>\lambda\right\rbrace)\leq
t\right\rbrace.
$$
We denote by  $f^{**}$ the maximal decreasing rearrangement of
$f$: for every $t>0$,
$$
f^{**}(t)=\frac{1}{t}\int_{0}^{t}f^{*}(s) ds.
$$
\end{dfn}
 It is known that $(\mathcal{M}f)^{*}\sim f^{**}$ and $\mu (\left\lbrace x:\, |f(x)|>f^{*}(t)\right\rbrace)\leq t$ for all $t>0$.
We refer to \cite{bennett}, \cite{bergh}, \cite{calderon2} for other properties of $f^{*}$ and $f^{**}$.

To end with this subsection let us quote the following theorem (\cite{holmstedt}):
\begin{thm}\label{H} Let $(X,\mu)$ be a measure space where $\mu$ is a non-atomic 
positive measure. Take $0<p_{0}<p_{1}<\infty$. Then
$$
K(f,t,L_{p_{0}},L_{p_{1}})\sim\left(\int_{0}^{t^{\alpha}}(f^{*}(u))^{p_{0}}du\right)^\frac{1}{p_{0}}+t\left(\int_{t^{\alpha}}^{\infty}(f^{*}(u))^{p_{1}}du\right)^\frac{1}{p_{1}},
$$
where $\frac{1}{\alpha}=\frac{1}{p_{0}}-\frac{1}{p_{1}}$.
\end{thm}
\
\subsection{Sobolev spaces associated to a weight $V$}
For the definition of the non-homogeneous Sobolev spaces $W_{p,V}^{1}$ and the homogeneous one $\dot{W}_{p,V}^{1}$ see the introduction. We begin showing that $W_{\infty,V}^{1}$ and $\dot{W}_{p,V}^{1}$ are Banach spaces.
\begin{prop}$W_{\infty,V}^{1}$ equipped with  the norm
$$
\|f\|_{W_{\infty,V}^{1}}= \|f\|_{\infty}+\|\,|\nabla f|\,\|_{\infty}+ \|Vf\|_{\infty}
$$
is a Banach space.
\end{prop}
\begin{proof} Let $(f_{n})_{n}$ be a Cauchy sequence in $W_{\infty,V}^{1}$. Then it is a Cauchy sequence in $W_{\infty}^{1}$ and converges to $f$ in $W_{\infty}^{1}$. Hence $Vf_{n}\rightarrow Vf\; \mu-a.e.$. On the other hand, $Vf_{n}\rightarrow g$ in $L_{\infty}$, then $\mu-a.e.$ The unicity of the limit gives us $g=Vf$.
\end{proof}
%\begin{prop}
%Let $M$ be a complete Riemannian manifold. Then $C^{\infty}_{0}(M)$ is dense in $W_{p,V}^{1}(M)$.
%\end{prop}
%\begin{proof} The proof goes as in \cite{aubin1} for the classical Sobolev spaces on Riemannian manifolds.
%\end{proof}
\
\begin{prop} Assume that $M$ satisfies $(D)$ and admits a Poincar\'{e} inequality $(P_{s})$ for some $1\leq s<\infty$ and that $V\in A_{\infty}$. Then, for $s\leq p\leq \infty$, $\dot{W}_{p,V}^{1}$ equipped with the norm
$$
\|f\|_{ \dot{W}_{p,V}^{1}}=\|\,|\nabla f|\,\|_{p}+\|Vf\|_{p}
$$
is a Banach space.
\end{prop} 
\begin{proof} Let $(f_{n})_{n}$ be a Cauchy sequence in $\dot{W}_{p,V}^{1}$. There exist a sequence of functions $(g_{n})_{n}$ and a sequence of scalar $(c_{n})_{n}$ with $g_{n}=f_{n}-c_{n}$ converging to a function $g$ in $L_{p, loc}$ and $\nabla g_{n}$ converging to $\nabla g$ in $L_{p}$ (see \cite{goldshtein}). Moreover, since $(Vf_{n})_{n}$ is a Cauchy sequence in $L_{p}$, it converges to a function $h$ $\,\mu-a.e.$. Lemma \ref{FP} in section 3 below yields
$$
\int_{B}\left(|\nabla(f_{n}-f_{m})|^{s}+|V(f_{n}-f_{m})|^{s}\right)d\mu\geq C(B,V)\int_{B}|f_{n}-f_{m}|^{s}d\mu
$$
for all  ball $B$ of $M$. Thus, $(f_{n})_{n}$ is a Cauchy sequence in $L_{s,loc}$. Since $(f_{n}-c_{n})$ is also Cauchy in $L_{s,loc}$, the sequence of constants $(c_{n})_{n}$ is Cauchy in $L_{s,loc}$ and therefore converges to a constant $c$.
Take $f:=g+c$. We have $g_{n}+c=f_{n}-c_{n}+c \rightarrow f$ in $L_{p, loc}$. It follows that $f_{n}\rightarrow f$ in $L_{p, loc}$ and so $Vf_{n}\rightarrow Vf \;\mu-a.e.$. The unicity of the limit gives us $h=Vf$. Hence, we conclude that $f\in \dot{W}_{p,V}^{1}$ and $f_{n}\rightarrow f$ in $ \dot{W}_{p,V}^{1}$ which finishes the proof.
\end{proof}
In the following proposition we characterize the $W_{p,V}^{1}$. We have
\begin{prop} \label{DC} Let $M$ be a complete Riemannian manifold and let $V\in RH_{qloc}$ for some $1\leq q<\infty$. Consider, for $1\leq p<q$, 
$$H_{p,V}^{1}(M)=H_{p,V}^{1}=\left\lbrace f\in L_{p}:\; |\nabla f|\, \textrm{ and } Vf \in L_{p} \right\rbrace
$$ and equip it with the same norm as $W_{p,V}^{1}$. 
Then $C_{0}^{\infty}$ is dense in $H_{p,V}^{1}$ and hence $W_{p,V}^{1}=H_{p,V}^{1}$.
\end{prop}
\begin{proof} See the Appendix.
\end{proof}
Therefore, under the hypotheses of Proposition \ref{DC}, $W_{p,V}^{1}$ is the set of distributions $f\in L_{p}$ such that $|\nabla f|$ and $Vf$ belong to $L_{p}$.
\section{Principal tools}
\
We shall use the following form of Fefferman-Phong inequality. The proof is completely analogous to the one in $\mathbb{R}^{n}$ ( see \cite{shen1}, \cite{auscher3}): 
\begin{lem}\label{FP}(Fefferman-Phong inequality). Let $M$ be a complete Riemannian manifold satisfying $(D)$. Let $w\in A_{\infty}$ and $1\leq p <\infty$. We assume that $M$ admits also a Poincar\'{e} inequality $(P_{p})$. Then there is a constant $C>0$ depending only on the $A_{\infty}$ constant of $w$, $p$ and the constants in $(D),\,(P_{p})$, such that for all ball $B$ of radius $R>0$ and $u \in W_{p,loc}^{1}$
$$
\int_{B}(|\nabla u|^{p}+w|u|^{p})d\mu \geq %\frac{C m_{\beta}(R^{p}w_{B})}{R^{p}}
 C \min(R^{-p}, w_{B})\int_{B} |u|^{p}d\mu.
$$
\end{lem}
\begin{proof}
 Since $M$ admits a $(P_{p})$ Poincar\'{e} inequality, we have
$$
\int_{B}|\nabla u|^{p}d\mu \geq \frac{C}{R^{p}\mu(B)} \int_{B}\int_{B} |u(x)-u(y)|^{p}d\mu(x)d\mu(y).
$$
This and
$$
\int_{B}w|u|^{p}d\mu =\frac{1}{\mu(B)}\int_{B}\int_{B} w(x) |u(x)|^{p} d\mu(x)d\mu(y)
$$
lead easily to
$$
\int_{B}(|\nabla u|^{p}+w|u|^{p})d\mu \geq [\min(CR^{-p},w)]_{B}\int_{B}|u|^{p}d\mu.
$$

Now we use that $w\in A_{\infty}$: there exists $\varepsilon>0$, independent of $B$, such that $\,E=\left\lbrace x\in B: w(x)>\varepsilon w_{B}\right\rbrace $ satisfies $\mu(E)>\frac{1}{2}\mu(B)$. Indeed since $w\in A_{\infty}$ then there exists $1\leq p<\infty$ such that $w\in A_{p}$. Therefore,
$$
\frac{\mu(E^{c})}{\mu(B)}\leq C\left(\frac{w(E^{c})}{w(B)}\right)^{\frac{1}{p}}\leq C\epsilon^{\frac{1}{p}}.
$$ 
We take $\epsilon>0$ such that $C\epsilon^{\frac{1}{p}}<\frac{1}{2}$. We obtain then
 $$
[\min(CR^{-p},w)]_{B}\geq \frac{1}{2} \min(CR^{-p},\varepsilon w_{B})\geq C' \min(R^{-p},w_{B}).
$$
This proves the desired inequality and finishes the proof.
\end{proof}
\
We proceed to establish two versions of a Calder\'{o}n-Zygmund decomposition:
\begin{prop}\label{CZ} Let $M$ be a complete non-compact Riemannian manifold satisfying $(D)$. Let $V\in RH_{q}$, for some $1< q<\infty$ and assume that $M$ admits a Poincar\'{e} inequality $(P_{s})$ for some $1\leq s<q$. Let $f\in W_{p,V}^{1}$, $s\leq p<q$, and $\alpha>0$. Then one can find a collection of balls $(B_{i})$, functions $g\in W_{q,V}^{1}$ and $b_{i}\in W_{s,V}^{1}$ with the following properties
\begin{equation}\label{1}
f=g+\sum_{i}b_{i}
\end{equation}
\begin{equation}\label{2}
\int_{\cup_{i}B_{i}} T_{q}g\, d\mu \leq C \alpha^{q} \mu(\cup_{i}B_{i})
\end{equation}
\begin{equation}\label{3}
\supp\, b_{i}\subset B_{i},\; \int_{B_{i}}T_{s}b_{i}\,d\mu\leq C\alpha^{s} \mu(B_{i})
 \end{equation}
 \begin{equation}\label{4}
 \sum_{i}\mu(B_{i})\leq \frac{C}{\alpha^{p}} \int_{M} T_{p}f\, d\mu
 \end{equation}
 \begin{equation}\label{5}
 \sum_{i}\ind_{B_{i}}\leq N
 \end {equation}
 where $N,\,C$ depend only on the constants in $(D)$, $(P_{s})$, $p$ and the $RH_{q}$ constant of $V$. Denote $T_{r}f= |f|^{r}+|\nabla f|^{r}+|Vf|^{r}$ for $1\leq r<\infty$.
 \end{prop}
 \begin{proof}
  Let $f\in W_{p,V}^{1}$, $\alpha>0$. Consider $\Omega =\left\lbrace x\in M: \mathcal{M}T_{s}f(x)>\alpha^{s}\right\rbrace$. If $\Omega=\emptyset$, then set 
  $$
   g=f, \; b_{i}=0\; \textrm{for all }i
  $$ 
  so that (\ref{2}) is satisfied thanks to the Lebesgue differentiation theorem. Otherwise the maximal theorem --Theorem \ref{MIT}-- and $p\geq s$ give us that 
  \begin{equation}\label{mO}
  \mu(\Omega)\leq \frac{C}{\alpha^{p}}\int_{M}T_{p}f \,d\mu<\infty.
  \end{equation}
   In particular $\Omega\neq M$ as $\mu(M)=\infty$. Let $F$ be the complement of $\Omega$. Since $\Omega$ is an open set distinct of $M$, let $(\underline{B_{i}})$ be a Whitney decomposition of $\Omega$ (\cite{coifman1}). That is, the balls  $\underline{B_{i}}$ are pairwise disjoint and there exist two constants $C_{2}>C_{1}>1$, depending only
on the metric, such that
\begin{itemize}
\item[1.] $\Omega=\cup_{i}B_{i}$ with $B_{i}=
C_{1}\underline{B_{i}}$ and the balls $B_{i}$ have the bounded overlap property;
\item[2.] $r_{i}=r(B_{i})=\frac{1}{2}d(x_{i},F)$ and $x_{i}$ is 
the center of $B_{i}$;
\item[3.] each ball $\overline{B_{i}}=C_{2}\underline{B_{i}}$ intersects $F$ ($C_{2}=4C_{1}$ works).
\end{itemize}
For $x\in \Omega$, denote $I_{x}=\left\lbrace i:x\in B_{i}\right\rbrace$. By the bounded overlap property of the balls $B_{i}$, we have that $\sharp I_{x} \leq N$. Fixing $j\in I_{x}$ and using the properties of the $B_{i}$'s, we easily see that $\frac{1}{3}r_{i}\leq r_{j}\leq 3r_{i}$ for all $i\in I_{x}$. In particular, $B_{i}\subset 7B_{j}$ for all $i\in I_{x}$.

Condition (\ref{5}) is nothing but the bounded overlap property of the $B_{i}$'s  and (\ref{4}) follows from (\ref{5}) and  (\ref{mO}). Note that $V\in RH_{q}$ implies $V^{q}\in A_{\infty}$ because there exists $\epsilon>0$ such that $V\in RH_{q+\epsilon}$ and hence $V^q\in RH_{1+\frac{\epsilon}{q}}$. Proposition \ref{CR} shows then that $V^{s}\in RH_{\frac{q}{s}}$. Applying Lemma \ref{FP} we get
\begin{equation}\label{t}
\int_{B_{i}} (|\nabla f|^{s}+|Vf|^{s})d\mu \geq C \min(V^{s}_{B_{i}},r_{i}^{-s})\int_{B_{i}}|f|^{s}d\mu.
\end{equation}
We declare $B_{i}$ of type 1 if $V_{B_{i}}^{s}\geq r_{i}^{-s}$ and of type 2 if $V_{B_{i}}^{s}< r_{i}^{-s}$. One should read $V_{B_{i}}^{s}$ as $(V^{s})_{B_{i}}$ but this is also equivalent to $(V_{B_{i}})^{s}$ since $V\in RH_{q}\subset RH_{s}$.

Let us now define the functions $b_{i}$. Let $(\chi_{i})_{i}$ be a partition of unity of $\Omega$ subordinated to the covering $(\underline{B_{i}})$, such that for all $i$, $\chi_{i}$ is a Lipschitz function supported in $B_{i}$ with
$\displaystyle\|\,|\nabla \chi_{i}|\, \|_{\infty}\leq
\frac{C}{r_{i}}$. To this end it is enough to choose $\displaystyle\chi_{i}(x)=
\psi(\frac{C_{1}d(x_{i},x)}{r_{i}})\Bigl(\sum_{k}\psi(\frac{C_{1}d(x_{k},x)}{r_{k}})\Bigr)^{-1}$, where $\psi$ is a smooth function, $\psi=1$ on $[0,1]$, $\psi=0$
on $[\frac{1+C_{1}}{2},+\infty[$ and $0\leq \psi\leq 1$. 
Set 
\begin{equation*}
b_{i}= \begin{cases}
f\chi_{i} \;\textrm{ if }\, B_{i}\; \textrm{ of type 1},
\\ 
(f-f_{B_{i}})\chi_{i}\; \textrm{  if}\, B_{i} \textrm{ of type 2}.
\end{cases}
\end{equation*}
 Let us estimate $\int_{B_{i}}T_{s}b_{i}\,d\mu$. We distinguish two cases:
\begin{itemize}
\item[1.] If $B_{i}$ is of type $2$, then 
\begin{align*}
\int_{B_{i}} |b_{i}|^{s} d\mu
&=\int_{B_{i}} |(f-f_{B_{i}})\chi_{i}|^{s} d\mu
\\
&\leq
C\left(\int_{B_{i}}|f|^{s}d\mu+\int_{B_{i}}|f_{B_{i}}|^{s} d\mu\right)
\\
&\leq C\int_{B_{i}}|f|^{s} d\mu
\\
&\leq C\int_{\overline{B_{i}}}|f|^{s} d\mu
\\
&\leq C \alpha^{s} \mu(\overline{B_{i}})
\\
&\leq C \alpha^{s} \mu(B_{i})
\end{align*}
where we used that $\overline{B_{i}} \cap F \neq
\emptyset$ and the property $(D)$. 
The Poincar\'{e} inequality $(P_{s})$ gives us
\begin{align*}
\int_{B_{i}}|\nabla b_{i}|^{s}d\mu &\leq C \int_{B_{i}}|\nabla f|^{s}d\mu
\\
&\leq C \mathcal{M}T_{s}f(y) \mu(B_{i})
\\
&\leq C\alpha^{s} \mu(B_{i})
\end{align*}
as $y$ can be chosen in $F \cap \overline{B_{i}}$. 
Finally,
\begin{align*}
\int_{B_{i}} |Vb_{i}|^{s}d\mu &=\int_{B_{i}} |V(f-f_{B_{i}})\chi_{i}|^{s}d\mu
\\
&\leq \int_{B_{i}} |Vf|^{s}d\mu+\int_{B_{i}}|Vf_{B_{i}}|^{s}d\mu
\\
&\leq (|Vf|^{s})_{B_{i}}\mu(B_{i})+C (V^{s})_{B_{i}}(|f|^{s})_{B_{i}}\mu(B_{i})
\\
&\leq C \alpha^{s} \mu(B_{i})+\left(|\nabla f|^{s}+|Vf|^{s}\right)_{B_{i}}\mu(B_{i})
\\
&\leq C\alpha^{s} \mu(B_{i}).
\end{align*}
We used that $\overline{B_{i}} \cap F\neq \emptyset$, Jensen's inequality and (\ref{t}), noting that $B_{i}$ is of type 2.
\item[2.] If $B_{i}$ is of type 1, then
\begin{align*}
\int_{B_{i}} T_{s}b_{i}\,d\mu &\leq \int_{B_{i}} T_{s}f\,d\mu+ r_{i}^{-s}\int_{B_{i}}|f|^{s}d\mu
\\
&\leq C\int_{B_{i}} T_{s}f\,d\mu
\\
&\leq C\alpha^{s}\mu(B_{i})
\end{align*}
where we used that $\overline{B_{i}}\cap F\neq \emptyset$ and that $B_{i}$ is of type 1.
\end{itemize}

Set now $g=f-\sum_{i}b_{i}$, where the sum is over balls of both types and is locally finite by (\ref{5}). The function $g$ is defined  almost everywhere on $M$, $g=f$ on $F$ and $g=\sum{}^2 f_{B_{i}} \chi_{i}$ on $\Omega$ where $\sum{}^j$ means that we are summing over balls of type $j$. Observe that $g$ is a locally integrable function on $M$. Indeed, let $\varphi\in L_{\infty}$ with compact support. Since $d(x,F)\geq r_{i}$ for $x \in \supp\, \,b_{i}$, we obtain
\begin{equation*} \int\sum_{i}|b_{i}|\,|\varphi|\,d\mu \leq
\Bigl(\int\sum_{i}\frac{|b_{i}|}{r_{i}}\,d\mu\Bigr)\,\sup_{x\in
M}\Bigl(d(x,F)|\varphi(x)|\Bigr)\quad
\end{equation*}
and
\begin {align*}
\int \frac{|b_{i}|}{r_{i}}d\mu
&=\int_{B_{i}}\frac{|f-f_{B_{i}}|}{r_{i}}\chi_{i}\,d\mu
\\
&\leq \Bigl(\mu(B_{i})\Bigr)^{\frac{1}{s'}}
\Bigl(\int_{B_{i}}|\nabla f|^{s} d\mu\Bigr)^{\frac{1}{s}}
\\
&\leq C\alpha\mu(B_{i}).
\end{align*}
We used the H\"{o}lder inequality, $(P_{s})$ and that $\overline{B_{i}}\cap F\neq \emptyset$, $s'$ being the conjugate of $s$. Hence 
$ \displaystyle \int\sum_{i}|b_{i}||\varphi|d\mu \leq
C\alpha\mu(\Omega) \sup_{x\in M
}\Bigl(d(x,F)|\varphi(x)|\Bigr)$. Since $f\in L_ {1,loc}$, we conclude that $g\in L_{1,loc}$. (Note that since $b\in L_{1}$ in our case, we can say directly that $g\in L_{1,loc}$. However, for the homogeneous case --section 5-- we need this observation to conclude that $g\in L_{1,loc}$.)  It remains to prove (\ref{2}). Note that $\displaystyle \sum_{i}\chi_{i}(x)=1$ and $\displaystyle \sum_{i}\nabla\chi_{i}(x)=0$  for all $x\in \Omega$. A computation of the sum $\sum_{i}\nabla b_{i}$ leads us to
$$
\nabla g = (\nabla f)\ind_{F} +\sum{}^2 f_{B_{i}} \nabla \chi_{i}.
 $$
By definition of $F$ and the differentiation theorem, $|\nabla g|$ is bounded by $\alpha$ almost everywhere on $F$.  It remains to control $\|h_{2}\|_{\infty}$ where $h_{2}= \sum{}^2 f_{B_{i}}  \nabla \chi_{i}$. Set $h_{1}=\sum{}^1  f_{B_{i}}  \nabla \chi_{i}$. By already seen arguments  for type 1 balls,  $|f_{B_{i}} | \leq C\alpha r_{i} $. Hence,
$|h_{1}| \le C\sum{}^1\ \ind_{B_{i}} \alpha\leq CN\alpha$ and it suffices to show that $h=h_{1}+h_{2}$ is bounded by $C\alpha$. To see this, fix $x \in \Omega$. Let
$B_{j}$ be a Whitney ball containing
$x$. We may write 
$$
|h(x)| = \left|\sum_{i \in I_x} (f_{B_{i}} - f_{B_{j}}) \nabla\chi_{i}(x)\right| \leq C \sum_{i \in I_{x}} |f_{B_{i}} - f_{B_{j}} | r_i^{-1}.
$$
 Since $B_{i}\subset7B_{j}$ for all $i\in I_{x}$, the Poincar\'e inequality $(P_{s})$ and the definition of $B_{j}$ yield
$$
|f_{B_{i}} - f_{B_{j}} | \leq
 Cr_{j} \left((|\nabla f|^{s})_{7B_{j}}\right)^{\frac{1}{s}} \leq Cr_{j} \alpha.
 $$ 
 Thus $\|h\|_{\infty}\leq C\alpha$.
 
 Let us now estimate $\int_{\Omega} T_{q}g\,d\mu$.
 We have 
 \begin{align*}
 \int_{\Omega}|g|^{q}d\mu&=\int_{M}|(\sum{}^2 f_{B_{i}} \chi_{i})|^{q}d\mu
 \\
 &\leq C\sum{}^2 |f_{B_{i}}|^{q}\mu(B_{i})
 \\
 &\leq CN\alpha^{q}\mu(\Omega).
\end{align*}
We used the estimate
\begin{equation*}
(|f|_{B_{i}})^{s} \leq (|f|^{s})_{B_{i}}\leq  (\mathcal{M}T_{s}f)(y)\leq \alpha^{s}
\end{equation*}
as $y$ can be chosen in $F\cap \overline{B_{i}}$.
For $|\nabla g|$, we have
\begin{align*}
\int_{\Omega}|\nabla g|^{q}d\mu&=\int_{\Omega}|h_{2}|^{q}d\mu
\\
&\leq C\alpha^{q}\mu(\Omega). 
\end{align*}
Finally, since  by Proposition \ref{CR} $\,V^{s}\in RH_{\frac{q}{s}}$, we get
\begin{align*}
\int_{\Omega}V^{q}|g|^{q}d\mu &\leq \sum{}^2 \int_{B_{i}}V^{q}|f_{B_{i}}|^{q}d\mu
\\
&\leq C\sum{}^2 (V^{s}_{B_{i}}|f_{B_{i}}|^{s})^{\frac{q}{s}}\mu(B_{i}).
\end{align*}
By construction of the type 2 balls and by (\ref{t}) we have
$V^{s}_{B_{i}}|f_{B_{i}}|^{s}\leq V^{s}_{B_{i}}(|f|^{s})_{B_{i}}\leq C (|\nabla f|^{s}+|Vf|^{s})_{B_{i}}\leq C\alpha^{s}$.
Then
$\int_{\Omega}V^{q}|g|^{q}d\mu\leq C \sum{}^2\alpha^{q}\mu(B_{i}) \leq NC\alpha^{q}\mu(\Omega)$.

To finish the proof, we have to verify that $g\in W_{q,V}^{1}$. For that we just have to control $\int_{F}T_{q}g\,d\mu$. As $g=f$ on $F$, this readily follows from
\begin{align*}
\int_{F} T_{q}f d\mu&=\int_{F}(|f|^{q}+|\nabla f|^{q}+|Vf|^{q})d\mu 
\\
&\leq \int_{F}(|f|^{p}|f|^{q-p}+|\nabla f|^{p}|\nabla f|^{q-p}+|Vf|^{p}|Vf|^{q-p})d\mu \\
&\leq \alpha^{q-p} \|f\|_{W_{p,V}^{1}}^{p}.
\end{align*}
\end{proof}
\begin{rem} 1-It is a straightforward consequence from (\ref{3}) that $b_{i}\in W_{r,V}^{1}$ for all $1\leq r\leq s$ with $\|b_{i}\|_{W_{r,V}^1}\leq C\alpha \mu(B_{i})^{\frac{1}{r}}$.
\\
2-The estimate $\int_{F}T_{q}g\,d\mu$ above is too crude to be used in the interpolation argument. Note that (\ref2) only involves control of $T_{q}g$ on $\Omega=\cup_{i}B_{i}$. Compare with (\ref{egI}) in the next argument when $q=\infty$.
\end{rem} 
\begin{prop}\label{CZI} Let $M$ be a complete non-compact Riemannian manifold satisfying $(D)$. Let $V\in RH_{\infty}$ and  assume that $M$ admits a Poincar\'{e} inequality $(P_{s})$ for some $1\leq s<\infty$. Let $f\in W_{p,V}^{1}$, $s\leq p<\infty$, and $\alpha>0$. Then one can find a collection of balls $(B_{i})$, functions $b_{i}$ and a Lipschitz function $g$ such that the following properties hold:
\begin{equation}
f = g+\sum_{i}b_{i} \label{dfI}
\end{equation}
\begin{equation}
\|g\|_{W_{\infty,V}^{1}}\leq C\alpha \label{egI}
\end{equation}
\begin{equation}
\supp\, b_{i}\subset B_{i}, \, \forall\, 1\leq r\leq s\;\int_{B_{i}}T_{r}b_{i}\,d\mu\leq C\alpha^{r}\mu(B_{i})\label{eb}
\end{equation}
\begin{equation}
\sum_{i}\mu(B_{i})\leq \frac{C}{\alpha^{p}}\int T_{p}f\, d\mu
\label{eB}
\end{equation}
\begin{equation}
\sum_{i}\chi_{B_{i}}\leq N \label{rb}
\end{equation}
where $C$ and $N$  only depend on the constants in $(D)$, $(P_{s})$, $p$  and  the $RH_{\infty}$ constant of $V$.
\end{prop}
\begin{proof} The only difference between the proof of this proposition and that of Proposition \ref{CZ} is the estimation (\ref{egI}). Indeed, as we have seen in the proof of Proposition \ref{CZ}, we have $|\nabla g|\leq C\alpha $ almost everywhere. By definition of $F$ and the differentiation theorem, $(|g|+|Vg|)$ is bounded by $\alpha$ almost everywhere on $F$. We have also seen that for all $i$, $|f|_{B_{i}}\leq \alpha$.
% \begin{align*}
 %(f_{B_{i}})^{s}&\leq
  %C\left(\frac{1}{\mu(\overline{B_{i}})}\int_{\overline{B_{i}}}|f|d\mu \right)^{s}
%\\
% &\leq C\aver{\overline{B_{i}}}|f|^{s}d\mu
% \\
% &\leq C \mathcal{M}T_{s}f(y)
% \\
% &\leq C \alpha^{s}
% \end{align*}
 %where $y\in \overline{B_{i}}\cap F$ since $\overline{B_{i}}\cap F\neq \emptyset$. 
 %The second inequality follows from the fact that $(\mathcal{M}f)^{s}\leq \mathcal{M}f^{s}$ for $s\geq 1$.\\
%Hence $f_{B_{i}}\leq \alpha$.
 Fix $x\in \Omega$, then 
 \begin{align*}
 |g(x)|&=|\sum_{i\in I_{x}}f_{B_{i}}|
 \\
 &\leq \sum_{i\in I_{x}}|f_{B_{i}}|
 \\
 &\leq N\alpha.
 \end{align*}
It remains to estimate $|Vg|(x)$. We have
\begin{align*}
|Vg|(x)&\leq \sum{}^2_{i:x\in B_{i}} V(x) |f_{B_{i}}|
\\
&\leq C\sum{}^2_{i:x\in B_{i}} (V_{B_{i}}) |f_{B_{i}}|
\\
&\leq C\sum{}^2_{i:x\in B_{i}} \left((V^{s})_{B_{i}} (|f|^{s})_{B_{i}}\right)^{\frac{1}{s}}
\\
&\leq C\sum{}^2_{i:x\in B_{i}} (|\nabla f|^{s}+|Vf|^{s})_{B_{i}}^{\frac{1}{s}}
\\
&\leq NC\alpha
\end{align*}
where we used the definition of $RH_{\infty}$, and Jensen's inequality as $s\geq 1$. We used also (\ref{t}) and the bounded overlap property of the $B_{i}$'s.
\end{proof}
\section{Estimation of the $K$-functional in the non-homogeneous case}
Denote for $1\leq r<\infty$, $T_{r}f= |f|^{r}+|\nabla f|^{r}+|Vf|^{r}$, $T_{r*}f= |f|^{r*}+|\nabla f|^{r*}+|Vf|^{r*}$, $T_{r**}f= |f|^{r**}+|\nabla f|^{r**}+|Vf|^{r**}$. We have $tT_{r**}f(t)=\int_{0}^{t}T_{r*}f(u)du$ for all $t>0$.
\begin{thm}\label{EKI} Under the same hypotheses as in Theorem \ref{IS}, with $V\in RH_{\infty loc}$ and $1\leq r\leq  s<\infty$:
\begin{itemize}
 \item[1.] there exists $C_{1}>0$ such that for every $f \in W_{r,V}^{1}+W_{\infty,V}^{1}$ and $t>0$ 
 \begin{equation*}
K(f,t^{\frac{1}{r}},W_{r,V}^{1},W_{\infty,V}^{1})\geq C_{1}\left(\int_{0}^{t}T_{r*}f(u)du\right)^{\frac{1}{r}}\sim \left(tT_{r**}f(t)\right)^{\frac{1}{r}};
\end{equation*}
\item[2.]  for $s\leq p<\infty$, there is $C_2>0$ such that for every $f\in W_{p,V}^{1}$ and $t>0$ 
\begin{equation*}
K(f,t^{\frac{1}{r}},W_{r,V}^{1},W_{\infty,V}^{1})\leq C_{2}t^{\frac{1}{r}}\left(T_{s**}f(t)\right)^{\frac{1}{s}}.
\end{equation*}  
\end{itemize}
In the particular case when $r=s$, we obtain the upper bound of $K$ for every $f\in W_{r,V}^{1}+W_{\infty,V}^{1}$ and get therefore a true characterization of $K$.
 \end{thm}
 \begin{proof} We refer to \cite{badr1} for an analogous proof.
 \end{proof}
\begin{thm}\label{EK} We consider the same hypotheses as in Theorem \ref{IS} with $V\in RH_{qloc}$ for some $1<q<\infty$. Then
 %Then, there exist $C_{1},\,C_{2}>0$ such that for every $f\in W_{s,V}^{1}+W_{q,V}^{1}$, and all $t>0$ \begin{itemize}
\begin{itemize}
\item[1.] there exists $C_{1}$ such that for every $f \in W_{r,V}^{1}+W_{q,V}^{1}$ and $t>0$
$$
K(f,t, W_{r,V}^{1},W_{q,V}^{1})\geq C_{1}\left(t^{\frac{q}{q-r}}(T_{r**}f)^{\frac{1}{r}}(t^{\frac{qr}{q-r}})+t \left(\int_{t^{\frac{qr}{q-r}}}^{\infty}T_{r*}f(u)du\right)^{\frac{1}{r}}\right);
$$
\item[2.] for $s\leq p<q$,
 there is $C_{2}$ such that for every $f\in W_{p,V}^{1}$ and $t>0$
$$ K(f,t, W_{r,V}^{1},W_{q,V}^{1})\leq C_{2}\left(t^{\frac{q}{q-r}}(T_{s**}f)^{\frac{1}{s}}(t^{\frac{qr}{q-r}})+t\left(\int_{t^{\frac{qr}{q-r}}}^{\infty} \left(\mathcal{M}T_{s}f\right)^{*\frac{q}{s}}(u) du\right)^{\frac{1}{q}}\right).
$$
\end{itemize}
\end{thm}
 \begin{proof} In a first step we prove this theorem in the global case. This will help to understand the proof of the more general local case. 
\subsection{The global case}
Let $M$ be a complete Riemannian manifold satisfying $(D)$. Let $V\in RH_{q}$ for some $1<q<\infty$ and assume that $M$ admits  a Poincar\'{e} inequality $(P_{s})$ for some $1\leq s<q$. The principal tool to prove Theorem \ref{EK} in this case will be the Calder\'{o}n-Zygmund decomposition of Proposition \ref{CZ}.

 We prove the left inequality by applying Theorem \ref{H} with $p_{0}=r$ and $p_{1}=q$ which gives for all $f\in L_{r}+L_{q}$:
$$
K(f,t,L_{r},L_{q})\sim \left(\int_{0}^{t^{\frac{qr}{q-r}}}f^{*r}(u)du\right)^{\frac{1}{r}}+t \left(\int_{t^{\frac{qr}{q-r}}}^{\infty}f^{*q}(u)du\right)^{\frac{1}{q}}.
$$
Moreover, we have
$$
K(f,t, W_{r,V}^{1}, W_{q,V}^{1})\geq K(f,t,L_{r},L_{q})+K(|\nabla f|,t,L_{r},L_{q})+K( Vf,t,L_{r},L_{q})
$$
since the operator
$$
(I,\,\nabla,\, V): W_{l,V}^{1}\rightarrow L_{l}(M;\mathbb{C}\times TM\times \mathbb{C})
$$
is bounded for every $1\leq l\leq \infty$.
\\
Hence we conclude with
\begin{align*}
K(f,t,W_{r,V}^{1},W_{q,V}^{1})&\geq  C\left(\int_{0}^{t^{\frac{qr}{q-r}}}T_{r*}f(u)du\right)^{\frac{1}{r}}+Ct \left(\int_{t^{\frac{qr}{q-r}}}^{\infty}T_{q*}f(u)du\right)^{\frac{1}{q}}.
 \end{align*}

 We now prove item 2. Let $f\in W_{p,V}^{1},\,s\leq p<q$ and $t>0$. We consider the Calder\'{o}n-Zygmund decomposition of $f$ given by Proposition \ref{CZ} with $\alpha=\alpha(t)=(\mathcal{M}T_{s}f)^{*\frac{1}{s}}(t^{\frac{qr}{q-r}})$. Thus $f$ can be written as $f=b+g$ with $b= \sum\limits_{i}b_{i}$ where $(b_{i})_{i},\,g$ satisfy the properties of the proposition. For the $L_{r}$ norm of $b$ we  have 
 \begin{align*}
\| b \|_{r}^{r}&\leq \int_{M}(\sum_{i}
|b_{i}|)^{r}d\mu
\\
&\leq N \sum_{i}\int_{B_{i}}
|b_{i}|^{r}d\mu
\\
&\leq
C\alpha^{r}(t)\sum_{i}\mu(B_{i})
\\
&\leq N C\alpha^{r}(t)\mu(\Omega_{t}).
\end{align*}
This follows from the fact that
$\displaystyle \sum_{i}\chi_{B_{i}}\leq N$ and
$\Omega_{t}=\Omega=\underset{i}{\bigcup}B_{i}$. Similarly we get $\| \,|\nabla b|\,\|_{r}^{r}\leq
C\alpha^{r}(t)\mu(\Omega_{t})\,$ and $\,\|Vb\|_{r}^{r}\leq C\alpha^{r}(t)\mu(\Omega_{t})$. For $g$ we have $\|g\|_{W_{q,V}^{1}}\leq C\alpha(t) \mu(\Omega_{t})^{\frac{1}{q}}+\left(\int_{F_{t}}T_{q}f d\mu\right)^{\frac{1}{q}}$, where $F_{t}=F$ in the Proposition \ref{CZ} with this choice of $\alpha$.
 
Moreover, since $(\mathcal{M}f)^{*}\sim f^{**}$ and $(f+g)^{**}\leq f^{**}+g^{**}$, we obtain
$$
\alpha(t)=(\mathcal{M}T_{s}f)^{*\frac{1}{s}}(t^{\frac{qr}{q-r}})\leq C (T_{s**}f)^{\frac{1}{s}}(t^{\frac{qr}{q-r}}).
$$
 Notice that for every $t>0$, $\mu(\Omega_{t})\leq t^{\frac{qr}{q-r}}$. It comes that
\begin{equation}\label{K}
 K(f,t,W_{r,V}^{1},W_{q,V}^{1})\leq
Ct^{\frac{q}{q-r}}(T_{s**}f)^{\frac{1}{s}}(t^{\frac{qr}{q-r}})+Ct\left(\int_{F_{t}} T_{q}f d\mu\right)^{\frac{1}{q}}.
\end{equation} 
Let us estimate $\int_{F_{t}} T_{q}f d\mu$. Consider $E_{t}$ a measurable set such that 
$$
\Omega_{t}\subset E_{t}\subset\left\lbrace x: \mathcal{M}T_{s}f(x)\geq (\mathcal{M}T_{s}f)^{*}(t^{\frac{qr}{q-r}})\right\rbrace 
$$
 and $\mu(E_{t})=t^{\frac{qr}{q-r}}$. Remark that $\int_{E_{t}}(\mathcal{M}T_{s}f)^{l}d\mu=\int_{0}^{t^{\frac{qs}{q-s}}}(\mathcal{M}T_{s}f)^{*l}(u)du\,$ for $l\geq 1$ --see \cite{stein3},\,Chapter V, Lemma 3.17--. Denote $G_{t}:=E_{t}-\Omega_{t}$.
Then
\begin{align}\label{K+}
\int_{F_{t}} T_{q}f d\mu &=\int_{E_{t}^{c}} T_{q}f d\mu+\int_{G_{t}} T_{q}f d\mu \nonumber
\\
&\leq C \int_{t^{\frac{qr}{q-r}}}^{\infty}(\mathcal{M}T_{s}f)^{*\frac{q}{s}}(u)du+C\int_{G_{t}}(T_{s**}f)^{\frac{q}{s}}(t^{\frac{qr}{q-r}})d\mu \nonumber
\\
&\leq C \int_{t^{\frac{qr}{q-r}}}^{\infty}(\mathcal{M}T_{s}f)^{*\frac{q}{s}}(u)du+C\mu(E_{t})(T_{s**}f)^{\frac{q}{s}}(t^{\frac{qr}{q-r}}) \nonumber
\\
&=C\int_{t^{\frac{qr}{q-r}}}^{\infty}(\mathcal{M}T_{s}f)^{*\frac{q}{s}}(u)du+Ct^{\frac{qr}{q-r}}(T_{s**}f)^{\frac{q}{s}}(t^{\frac{qr}{q-r}}).
\end{align}
Combining (\ref{K}) and (\ref{K+}) we deduce that
 \begin{align*}
 K(f,t,W_{r,V}^{1},W_{q,V}^{1})&\leq
Ct^{\frac{q}{q-r}}(T_{s**}f)^{\frac{1}{s}}(t^{\frac{qr}{q-r}})+Ct\left(\int_{t^{\frac{qr}{q-r}}}^{\infty}(\mathcal{M}T_{s}f)^{*\frac{q}{s}}(u)du\right)^{\frac{1}{q}}
\end{align*}
which finishes the proof in that case. 
\subsection{The local case}
Let $M$ be a complete non-compact Riemannian manifold satisfying a local doubling property $(D_ {loc})$. Consider $V\in RH_{qloc}$ for some $1<q<\infty$ and assume that $M$ admits a local Poincar\'{e} inequality $(P_{sloc})$ for some $1\leq s<q$.

Denote by $\mathcal{M}_{E}$ the Hardy-Littlewood maximal operator relative to a measurable subset 
$E$ of $M$, that is, for  $x\,\in E$ and every $f$ locally integrable function on $M$:
$$
\displaystyle\mathcal{M}_{E}f(x)= \sup_{B:\,x\in
B}\frac{1}{\mu(B\cap E)}\int_{B \cap E}|f|d\mu
$$
where $B$ ranges over all open balls of $M$ containing $x$ and centered in $E$.
We say that a measurable subset $E$ of $M$ has the relative doubling property if there exists a constant $C_{E}$ such that for all $x\in E$ and $r>0$ we have
$$
\mu(B(x,2r)\cap E)\leq C_{E}\mu(B(x,r)\cap E).
$$
This is equivalent to saying that the metric measure space $(E,d/E,\mu/E)$ has the doubling property.
On such a set $\mathcal{M}_{E}$ is of weak type $(1,1)$ and bounded on $L^{p}(E,\mu),\,1<p\leq\infty$.

We now prove Theorem \ref{EK} in the local case. To fix ideas, we assume $r_{0}=5$, $r_{1}=8$, $r_{2}=2$. The lower bound of $K$ in item 1. is trivial (same proof as for the global case). It remains to prove the upper bound. 
For all $t>0$,
 take $\alpha=\alpha(t)=(\mathcal{M}T_{s}f)^{*\frac{1}{s}}(t^{\frac{qs}{q-s}})$.

Consider 
$$
\Omega=\left\lbrace x\in M:\mathcal{M}T_{s}f(x)>\alpha^{s}(t)\right\rbrace.
$$
We have $\mu(\Omega)\leq t^{\frac{qr}{q-r}}$. If $\Omega=M$  then
\begin{align*}
 \int_{M}T_{r}f\,d\mu
 &=\int_{\Omega}T_{r}f\,d\mu
\\
&\leq C\int_{0}^{\mu(\Omega)}T_{r*}f(l)dl
\\
&\leq C \int_{0}^{t^{\frac{qr}{q-r}}}T_{r*}f(l)dl
\\
&\leq Ct^{\frac{qr}{q-r}}(T_{r**}f)^{\frac{1}{r}}(t^{\frac{qr}{q-r}})
%\\
%&=C t^{\frac{qs}{q-s}}\left(|f|^{s**}+|\nabla f|^{s**}+|Vf|^{s**}\right)(t^{\frac{qs}{q-s}}).
\end{align*}
Therefore
$$
 K(f,t,W_{r,V}^{1},W_{q,V}^{1})\leq Ct^{\frac{q}{q-r}}(T_{s**}f)^{\frac{1}{s}}(t^{\frac{qr}{q-r}})
$$
since $r\leq s$.
We thus obtain item 2. in this case. 

Now assume $\Omega \neq M$.
Pick a countable set
$\left\lbrace x_{j}\right\rbrace _{j\in J} \subset  M,$ such that $ M=
\underset{j\in J}{\bigcup}B(x_{j},\frac{1}{2})$ and for all $x\in M$,
$x$ does not belong to more than $N_{1}$ balls $B^{j}:=B(x_{j},1)$.
Consider a $C^{\infty}$ partition of unity $(\varphi_{j})_{j\in J}$ subordinated to the balls $\frac{1}{2}B^{j}$ such that $0\leq
\varphi_{j}\leq 1,\,\\supp\,\,\varphi_{j}\subset B^{j}$  and
$\|\,|\nabla \varphi_{j}|\, \|_{\infty}\leq C$ uniformly with respect to $j$.
 Consider $f\in W_{p,V}^{1}$, $s\leq p<q$. Let $ f_{j}=f\varphi_{j}$ so that $ f=\sum_{j\in J}f_{j}$. We have for $j\in J$, $f_{j},\, Vf_{j}\in L_{p}$ and $\;\nabla f_{j}=f\nabla
\varphi_{j}+\nabla f \varphi_{j} \in L_{p}$. Hence $f_{j}\in W_{p}^{1}(B^{j})$. The balls $B^{j}$ satisfy  the relative doubling property with the constant independent of the balls $B^{j}$. This follows from the next lemma quoted from \cite{auscher2} p.947.
\begin{lem}\label{DB}
Let $M$ be a  complete Riemannian manifold satisfying
$(D_ {loc})$. Then the balls $B^{j}$ above, 
equipped with the induced distance and measure, satisfy the
relative doubling property $(D)$, with the doubling constant that may be chosen independently of $j$.
More precisely, there exists $C\geq0$ such that for all $j\in J$
\begin{equation}
 \mu(B(x,2R)\cap B^{j})\leq
C\,\mu(B(x,R)\cap B^{j})\label{DB1}
\quad\forall x\in B^{j},\,R>0,
\end{equation}
and
\begin{equation}
\mu(B(x,R))\leq C\mu(B(x,R)\cap B^{j})\label{DB2}
\quad \forall x\in B^{j},\,0<R\leq 2.
 \end{equation}
\end{lem}

\noindent Let us return to the proof of the theorem. 
 For any $x\in B^{j}$ we have
\begin{align}
 \mathcal{M}_{B^{j}}T_{s}f_{j}(x)
 &=\sup_{B:\,x\in B,\,R(B)\leq 2}\frac{1}{\mu(B^{j}\cap
B)}\int_{B^{j}\cap B}T_{s}f_{j}d\mu \nonumber
\\
&\leq \sup_{B:\,x\,\in B,\;R(B)\leq 2}C\frac{\mu(B)}{\mu(B^{j}\cap B)}\frac{1}{\mu(B)}\int_{B}T_{s}fd\mu \nonumber
\\
&\leq C\mathcal{M}T_{s}f(x).\label{MB}
\end{align}
where we used (\ref{DB2}) of Lemma \ref{DB}. Consider now 
$$
\Omega_{j}=\left\lbrace x\in
B^{j}:\mathcal{M}_{B^{j}}T_{s}f_{j}(x)>C\alpha^{s}(t)\right\rbrace
$$
 where $C$ is the constant in
(\ref{MB}). The set $\Omega_{j}$ is an open subset of $B^{j}$ then of $M$ and $\Omega_{j}\subset \Omega $ for all $j \in J$. 
For the $f_{j}$'s, and for all $t>0$, we have a Calder\'{o}n-Zygmund decomposition similar to the one done in Proposition \ref{CZ}:
there exist $b_{jk},\;g_{j}$ supported in $B^{j}$, and balls $(B_{jk})_{k}$ of $M$, contained in $\Omega_{j}$, such that
\begin{equation}
f_{j} = g_{j}+\sum_{k}b_{jk} \label{f1}
\end{equation}
\begin{equation}
\int_{\Omega_{j}} T_{q}g_{j}\,d\mu \leq C \alpha^{q}(t)\mu(\Omega_{j}) \label{eg1}
\end{equation}
\begin{equation}
\\supp\,\, b_{jk}\subset B_{jk},\,\forall 1\leq r\leq s\;  \int_{B_{jk}}T_{r}b_{jk}\, d\mu
\leq C\alpha^{r}(t)\mu(B_{jk})\label{eb1}
\end{equation}
\begin{equation}
\sum_{k}\mu(B_{jk})\leq C\alpha^{-p}(t)\int_{B^{j}} T_{p}f_{j}\,d\mu \label{emB1}
\end{equation}
\begin{equation}
\sum_{k}\chi_{B_{jk}}\leq N \label{eiB1}
\end{equation}
with $C$ and $N$ depending only on $q$, $p$ and the constant $C(r_{0}),C(r_{1}), C(r_{2})$ in
$(D_{loc})$ and $(P_{sloc})$ and the $RH_{qloc}$ condition of $V$, which is independent of $B^{j}$.
\\
The proof of this decomposition is the same as that of Proposition \ref{CZ}, taking for all $j\in J$ a Whitney decomposition  $(B_{jk})_{k}$ of $\Omega_{j}\neq M$ and using the doubling property for balls whose radii do not exceed $3<r_{0}$ and the Poincar\'{e} inequality for balls whose radii do not exceed $7<r_{1}$ and the $RH_{qloc}$ property of $V$ for balls whose radii do not exceed $1<r_{2}$.
By the above decomposition we can write $ f=\sum\limits _{j\in J}\sum\limits_{k}b_{jk}+\sum\limits _{j\in J}g_{j}=b+g$.
Let us now estimate $\|b\|_{W_{r,V}^{1}}$ and $\|g\|_{W_{q,V}^{1}}$. 
\begin{align*} 
\|b\|_{r}^{r}&\leq N_{1}N\sum_{j}
\sum_{k}\|b_{jk}\|_{r}^{r}
\\
&\leq C\alpha^{r}(t) \sum_{j}\sum_{k}(\mu(B_{jk}))
\\
&\leq NC\alpha^{r}(t) \Bigl(\sum_{j}\mu(\Omega_{j})\Bigr)
\\
&\leq N_{1}C\alpha^{r}(t) \mu(\Omega).
\end{align*}
We used the bounded overlap property of the $(\Omega_{j})_{j\in J}$'s and that of the 
$(B_{jk})_{k}$'s for all $j\in J$. 
It follows that $\| b\|_{r}\leq C\alpha(t)\mu(\Omega)^{\frac{1}{r}}$.
Similarly we get $\|\,|\nabla b|\, \|_{r}\leq C\alpha(t)\mu(\Omega)^{\frac{1}{r}}$ and $\|V b\|_{r}\leq C\alpha(t)\mu(\Omega)^{\frac{1}{r}}$.

For $g$ we have
\begin{align*}
 \int_{\Omega}|g|^{q}d\mu & \leq N \sum_{j}\int_{\Omega_{j}}|g_{j}|^{q}d\mu
 \\
 &\leq NC\alpha^{q}(t)\sum_{j}\mu(\Omega_{j})
 \\
 &\leq N_{1}NC\alpha^{q}(t)\mu(\Omega).
\end{align*}
Analogously $\int_{\Omega}|\nabla g|^{q}d\mu\leq
C\alpha^{q}(t)\mu(\Omega)$ and $\int_{\Omega}|Vg|^{q}d\mu \leq C\alpha^{q}(t)\mu(\Omega)$. Noting that $g\in W_{q,V}^{1}$ --same argument as in the proof of the global case--, it follows that
\begin{align*}
 K(f,t,W_{r,V}^{1},W_{q,V}^{1})
&\leq \| b \|_{W_{r}^{1}}+t\|g\|_{W_{q}^{1}}
\\
&\leq C
\alpha(t)\mu(\Omega)^{\frac{1}{r}}+Ct\alpha(t)\mu(\Omega)^{\frac{1}{q}}+t\left(\int_{F_{t}}T_{q}f d\mu\right)^{\frac{1}{q}}
\\
&\leq Ct^{\frac{q}{q-r}}(T_{s**}f)^{\frac{1}{s}}(t^{\frac{qr}{q-r}})+t\left(\int_{t^{\frac{qr}{q-r}}}^{\infty}(\mathcal{M}T_{s}f)^{*\frac{q}{s}}(u)du\right)^{\frac{1}{q}}.
\end{align*}
Thus, we get the desired estimation for $f\in W_{p,V}^{1}$.
% and the proof of the general case for $f\in W_{s,V}^{1}+W_{q,V}^{1}$ goes as in the global case using the argument presented above.
\end{proof}
\section{Interpolation of non-homogeneous Sobolev spaces}
\begin{proof}[Proof of Theorem \ref{IS}]  The proof of the case when $V\in RH_{\infty loc}$ is the same as the one in section 4 in \cite{badr1}.
Consider now $V\in RH_{qloc}$ for some $1<q<\infty$. For $1\leq r\leq s< p<q$, we define the interpolation space $W_{p,r,q,V}^{1}(M)=W_{p,r,q,V}^{1}$ between $W_{r,V}^{1}$ and $W_{q,V}^{1}$ by
$$
W_{p,r,q,V}^{1}=(W_{r,V}^{1},W_{q,V}^{1})_{\frac{q(p-r)}{p(q-r)},p}.
$$
We claim that $W_{p,r,q,V}^{1}=W_{p,V}^{1}$ with equivalent norms. Indeed, let $f\in W_{p,r, q,V}^{1}$. We have 
\begin{align*}
\|f\|_{\frac{q(p-r)}{p(q-r)},p}=&\left\lbrace\int_{0}^{\infty} \left(t^{\frac{q(r-p)}{p(q-r)}}K(f,t,W_{r,V}^{1},W_{q,V}^{1})\right)^{p}\frac{dt}{t}\right\rbrace^{\frac{1}{p}}
\\
&\geq \left\lbrace \int_{0}^{\infty}\left(t^{\frac{q(r-p)}{p(q-r)}}t^{\frac{q}{q-r}}(T_{r**}f)^{\frac{1}{r}}(t^{\frac{qr}{q-r}})\right)^{p}\frac{dt}{t}\right\rbrace ^{\frac{1}{p}}
\\
&=\left\lbrace\int_{0}^{\infty}t^{\frac{qr}{q-r}-1}(T_{r**}f)^{\frac{p}{r}}(t^{\frac{qr}{q-r}})dt\right\rbrace^{\frac{1}{p}}
\\
&= \left\lbrace \int_{0}^{\infty}(T_{r**}f)^{\frac{p}{r}}(t)dt\right\rbrace^{\frac{1}{p}}
\\
&\geq \|f^{r**}\|_{\frac{p}{r}}^{\frac{1}{r}}+\| \,|\nabla f|^{r**}\|_{\frac{p}{r}}^{\frac{1}{r}}+\|\,|Vf|^{r**}\|_{\frac{p}{r}}^{\frac{1}{r}}
\\
&\sim \|f^{r}\|_{\frac{p}{r}}^{\frac{1}{r}}+\|\,|\nabla f|^{r}\|_{\frac{p}{r}}^{\frac{1}{r}}+\|\,|Vf|^{r}\|_{\frac{p}{r}}^{\frac{1}{r}}
\\
&= \|f\|_{W_{p,V}^{1}}
\end{align*}
where we used that for $l>1$, $\|f^{**}\|_{l}\sim \|f\|_{l}$.
Therefore $W_{p,r,q,V}^{1}\subset W_{p,V}^{1}$, with $\|f\|_{\frac{q(p-r)}{p(q-r)},p}\geq C \|f\|_{W_{p,V}^{1}}$.
 
On the other hand, let $f\in W_{p,V}^{1}$. By the Calder\'{o}n-Zygmund decomposition of Proposition \ref{CZ}, $f\in W_{r,V}^{1}+W_{q,V}^{1}$. Next, 
\begin{align*}
\|f\|_{\frac{q(p-r)}{p(q-r)},p}& \leq C\left\lbrace \int_{0}^{\infty} \left(t^{\frac{q(r-p)}{p(q-r)}}t^{\frac{q}{q-r}}(T_{s**}f)^{\frac{1}{s}}(t^{\frac{qr}{q-r}})\right)^{p}\frac{dt}{t}\right\rbrace ^{\frac{1}{p}}
\\
&+C\left\lbrace\int_{0}^{\infty}\left(t^{\frac{q(r-p)}{p(q-r)}}t\left(\int_{t^{\frac{qr}{q-r}}}^{\infty }(\mathcal{M}T_{s}f)^{*\frac{q}{s}}(u)du\right)^{\frac{1}{q}}\right)^{p}\frac{dt}{t}\right\rbrace ^{\frac{1}{p}}
\\
&=I+\,II.
\end{align*}
Using the same computation as above, we conclude that
\begin{align*}
I&\leq C \left\lbrace \int_{0}^{\infty}(T_{s**}f)^{\frac{p}{s}}(t)dt\right\rbrace^{\frac{1}{p}}
\\
&\leq C\|f\|_{W_{p,V}^{1}}.
\end{align*}
It remains to estimate II. We have
\begin{align*}
II&\leq C\left\lbrace \int_{0}^{\infty} t^{\frac{q(r-p)}{q-r}}t^{p}\left(\int_{t^{\frac{qr}{q-r}}}^{\infty}(\mathcal{M}T_{s}f)^{*\frac{q}{s}}(u)du\right)^{\frac{p}{q}}\frac{dt}{t} \right\rbrace^{\frac{1}{p}}
\\
&\leq C\left\lbrace \int_{0}^{\infty} t^{-\frac{p}{q}}\left(\int_{t}^{\infty}\left(u(\mathcal{M}T_{s}f)^{*\frac{q}{s}}(u)\right)\frac{du}{u}\right)^{\frac{p}{q}}dt\right\rbrace^{\frac{1}{p}}
\\
&\leq C\left\lbrace \int_{0}^{\infty} t^{-\frac{p}{q}}\left(\int_{t}^{\infty}\left(u(\mathcal{M}T_{s}f)^{*\frac{q}{s}}(u)\right)^{\frac{p}{q}}\frac{du}{u}\right)dt\right\rbrace^{\frac{1}{p}}
\\
&\leq \frac{C}{1-\frac{p}{q}}\left\lbrace \int_{0}^{\infty} t^{-\frac{p}{q}}(t(t^{\frac{p}{q}-1}(\mathcal{M}T_{s}f)^{*\frac{p}{s}}(t)))dt\right\rbrace^{\frac{1}{p}}
\\
&= C\|(\mathcal{M}T_{s}f)^{*}\|_{\frac{p}{s}}^{\frac{1}{s}}
\\
&\leq C\|\mathcal{M}T_{s}f\|_{\frac{p}{s}}^{\frac{1}{s}}
\\
&\leq C\|T_{s}f\|_{\frac{p}{s}}^{\frac{1}{s}}
\\
&\leq C\|f\|_{W_{p,V}^{1}}.
\end{align*}
We used the monotonicity of $(\mathcal{M}T_{s}f)^{*}$ together with $\frac{p}{q}<1$, the following Hardy inequality 
$$
\int_{0}^{\infty}\left[\int_{t}^{\infty}g(u)du\right]t^{l-1}dt\leq \left(\frac{1}{l}\right)\int_{0}^{\infty}\left[ug(u)\right]u^{l-1}du
$$
for $l=1-\frac{p}{q}>0$, the fact that $\|g^{*}\|_{l}\sim\|g\|_{l}$ for all $l\geq 1$ and Theorem \ref{MIT}.
Therefore, $W_{p,V}^{1}\subset W_{p,r,q,V}^{1}$ with $\|f\|_{ \frac{q(p-r)}{p(q-r)},p}\leq C\|f\|_{W_{p,V}^{1}}$.
\end{proof}
Let $A_{V}=\left\lbrace q\in ]1,\infty]: V\in RH_{qloc}\right\rbrace$ and $q_{0}=\sup\limits A_{V}$,  $B_{M}=\left\lbrace s\in[1,q_{0}[: (P_{sloc}) \textrm{ holds }\right\rbrace$ and $s_{0}=\inf B_{M}$.
\begin{cor}\label{CI} For all $p,\,p_{1},\,p_{2}$ such that $ 1\leq  p_{1} <p<p_{2}<q_{0} $ with $p>s_0$, $W_{p,V}^{1}$ is a real interpolation space between $W_{p_{1},V}^{1}$ and $W_{p_{2},V}^{1}$.
\end{cor}
\begin{proof}
Since $p_{2}<q_{0}$, item 1. of Proposition \ref{CRO} gives us that $V\in RH_{p_{2}loc}$. Therefore, Theorem \ref{IS} yields the corollary. (We could prove this corollary also using the reiteration theorem.)
%\item[2.]{Case when $1\leq p_{1}\leq s_{0}$}. Let $0<\theta<1$ such that $\frac{1}{p}=\frac{1-\theta}{p_1}+\frac{\theta}{p_{2}}$ and $\theta'=\theta (1-\frac{p_1}{p_{2}})=1-\frac{p_1}{p}$. The reiteration theorem --\cite{bennett}, Theorem 2.4 p.110-- applied only to the second exponent yields
%\begin{align*}
%(W_{p_{1},V}^{1}  ,W_{p_{2},V}^{1})_{\theta,p}&=(W_{p_{1},V}^{1},W_{p_{2},p_{1},V}^{1})_{\theta,p}\\
%&= (W_{p_1,V}^{1} ,W_{p_{2},V}^{1})_{\theta',p}
% \\
% &=W_{p,p_1,V}^{1}
% \\
% &= W_{p,V}^{1}.
% \end{align*}
% \end{itemize}
 \end{proof}
 \section{Interpolation of homogeneous Sobolev spaces} 
 Denote for $1\leq r<\infty$, $\dot{T}_{r}f= |\nabla f|^{r}+|Vf|^{r}$, $\dot{T}_{r*}f= |\nabla f|^{r*}+|Vf|^{r*}$ and $\dot{T}_{r**}f=|\nabla f|^{r**}+|Vf|^{r**}$. For the estimation of the functional $K$ for homogeneous Sobolev spaces we have the corresponding results:
\begin{thm}\label{EKH} Under the hypotheses of Theorem \ref{IHS} with $q<\infty$:
\begin{itemize}
\item[1.] there exists $C_{1}$ such that for every $f \in\dot{W}_{r,V}^{1}+\dot{W}_{q,V}^{1}$ and $t>0$
$$
 K(f,t,\dot{W}_{r,V}^{1},\dot{W}_{q,V}^{1})\geq C_{1}\left\lbrace\left(\int_{0}^{t^{\frac{qr}{q-r}}}\dot{T}_{r*}f(u)du \right)^{\frac{1}{r}}+t \left(\int_{t^{\frac{qr}{q-r}}}^{\infty}\dot{T}_{q*}f(u)du\right)^{\frac{1}{q}}\right\rbrace;
$$
\item[2.] for $s\leq p<q$, there exists $C_{2}$ such that for every $f\in \dot{W}_{p,V}^{1}$ and $t>0$
$$ 
K(f,t,\dot{W}_{r,V}^{1},\dot{W}_{q,V}^{1})\leq C_{2}\left\lbrace \left(\int_{0}^{t^{\frac{qr}{q-r}}}\dot{T}_{s*}f(u)du \right)^{\frac{1}{s}}+t\left(\int_{t^{\frac{qr}{q-r}}}^{\infty}\left(\mathcal{M}\dot{T}_{s}f\right)^{*\frac{q}{s}}(u) du\right)^{\frac{1}{q}}\right\rbrace.
$$
\end{itemize}
\end{thm}
\begin{thm}\label{EKHI}Under the hypotheses of Theorem \ref{IHS} with $V\in RH_{\infty}$:
\begin{itemize}
\item[1.] there exists $C_{1}$ such that for every $f \in\dot{W_{r,V}^{1}}+\dot{W}_{\infty,V}^{1}$ and $t>0$
$$ K(f,t^\frac{1}{r},\dot{W}_{r,V}^{1},\dot{W}_{\infty,V}^{1})\geq C_{1}t^{\frac{1}{r}}(\dot{T}_{r**}f)^{\frac{1}{r}}(t);
$$
\item[2.] for $ s\leq p<\infty$, there exists $C_{2}$ such that for every $f\in \dot{W}_{p,V}^{1}$  and every $t>0$
$$ 
K(f,t^{\frac{1}{r}},\dot{W}_{r,V}^{1},\dot{W}_{\infty,V}^{1})\leq C_{2} t^{\frac{1}{r}}(\dot{T}_{s**}f)^{\frac{1}{s}}(t).
$$
\end{itemize}
\end{thm}
 Before we prove Theorems \ref{EKH}, \ref{EKHI} and \ref{IHS}, we give two versions of a Calder\'{o}n-Zygmund decomposition.
 \begin{prop}\label{CZH1} Let $M$ be a complete non-compact Riemannian manifold satisfying $(D)$. Let $1\leq q <\infty$ and $V\in RH_{q}$. Assume that $M$
admits a  Poincar\'{e} inequality $(P_{s})$ for some $1\leq s<q$. Let $ s \leq p < q$ and consider $ f \in \dot{W}_{p,V}^{1}$ and $\alpha
>0$.
Then there exist a collection of balls $(B_{i})_{i}$, 
functions $b_{i}\in \dot{W}_{r,V}^{1}$ for $1\leq r\leq s$ and a function $g\in \dot{W}_{q,V}^{1}$ 
such that the following properties hold:
\begin{equation}
f = g+\sum_{i}b_{i} \label{dfn}
\end{equation}
\begin{equation}
\int_{\cup_{i}B_{i}}\dot{T}_{q}g\,d\mu\leq C\,\alpha^{q}\mu(\cup_{i}B_{i}) \label{egn}
\end{equation}
\begin{equation}
\\supp\,\, b_{i}\subset B_{i} \,\textrm{ and } \;\forall 1\leq r\leq s
 \int_{B_{i}}\dot{T}_{r}b_{i}\,d\mu
 \leq C\alpha^{r}\mu(B_{i})\label{ebn}
\end{equation}
\begin{equation}
\sum_{i}\mu(B_{i})\leq C\alpha^{-p}\int \dot{T}_{p}f\,d\mu
\label{eBn}
\end{equation}
\begin{equation}
\sum_{i}\chi_{B_{i}}\leq N \label{rbn}
\end{equation}
with $C$ and $N$ depending only on $q$, $s$ and the constants in
$(D)$, $(P_{s})$ and the $RH_{q}$ condition.
\end{prop}
 \begin{prop}\label{CZH2} Let $M$ be a complete non-compact Riemannian manifold satisfying $(D)$. Consider $V\in RH_{\infty}$. Assume that $M$
admits a  Poincar\'{e} inequality $(P_{s})$  for some $1\leq s<\infty$. 
Let $ s \leq p < \infty$, $ f \in \dot{W}_{p,V}^{1}$ and $\alpha
>0$. Then there exist a collection of balls $(B_{i})_{i}$, 
functions $b_{i}$ and a function $g$ 
such that the following properties hold :
\begin{equation}
f = g+\sum_{i}b_{i} \label{dfn}
\end{equation}
\begin{equation}
\dot{T}_{1}g\leq C\alpha
\quad \mu-a.e. \label{egn}
\end{equation}
\begin{equation}
\\supp\,\, b_{i}\subset B_{i} \,\textrm{ and }\;\forall 1\leq r\leq s\,
 \int_{B_{i}}\dot{T}_{r}b_{i}
d\mu \leq C\alpha^{r}\mu(B_{i})\label{ebn}
\end{equation}
\begin{equation}
\sum_{i}\mu(B_{i})\leq C\alpha^{-p}\int \dot{T}_{p} f\,d\mu
\label{eBn}
\end{equation}
\begin{equation}
\sum_{i}\chi_{B_{i}}\leq N \label{rbn}
\end{equation}
with $C$ and $N$ depending only on $q$, $p$ and the constant in
$(D)$, $(P_{s})$ and the $RH_{\infty}$ condition.
\end{prop}
The proof of these two decompositions goes as in the case of non-homogeneous Sobolev spaces, but taking 
$\Omega=\left\lbrace x\in M:\mathcal{M}\dot{T}_{s} f(x)>\alpha^{s}\right\rbrace$ as $\|f\|_{p}$ is not under control. We note that in the non-homogeneous case, we used that $f\in L_{p}$ only to control  $b\in L_{r}$ and $g\in L_{\infty}$ when $V\in RH_{\infty}$ and $\int_{\Omega}|g|^{q}d\mu$ when $V\in RH_{q}$ and $q<\infty$.
\begin{proof}[Proof of Theorem \ref{EKH} and \ref{EKHI}]
We refer to \cite{badr1} for the proof of Theorem \ref{EKHI}. The proof of item 1. of Theorem \ref{EKH} is the same as in the non-homogeneous case. Let us turn to inequality 2. Consider $f\in \dot{W}_{p,V}^{1}$, $t>0$ and $\alpha(t)=(\mathcal {M}\dot{T}_{s}f)^{*\frac{1}{s}}(t^{\frac{qr}{q-r}})$. By the Calder\'{o}n-Zygmund decomposition with $\alpha=\alpha(t)$, $f$ can be written $f=b+g$ with $\|b\|_{\dot{W}_{r,V}^{1}}\leq C\alpha(t) \mu(\Omega)^{\frac{1}{r}}$ and $\int_{\Omega} \dot{T}_{q}g d\mu\leq C\alpha^{q}(t)\mu(\Omega)$. Since we have $\mu(\Omega)\leq t^{\frac{qr}{q-r}}$, we get then as in the non-homogeneous case
$$
K(f,t,\dot{W}_{r,V}^{1},\dot{W}_{q,V}^{1})\leq Ct^{\frac{q}{q-r}}(\dot{T}_{s**}f)^{\frac{1}{s}}(t^{\frac{qr}{q-r}})+Ct\left(\int_{t^{\frac{qr}{q-r}}}^{\infty}(\mathcal{M}\dot{T}_{s} f)^{*\frac{q}{s}}(u)du \right)^{\frac{1}{q}}.
$$
\end{proof} 
\begin{proof}[Proof of Theorem \ref{IHS}] We refer to \cite{badr1} when $q=\infty$. When $q<\infty$, the proof follows directly from Theorem \ref{EKH}. Indeed, item 1. of Theorem \ref{EKH} gives us that  $$
(\dot{W}_{r,V}^{1},\dot{W}_{q,V}^{1})_{\frac{q(p-r)}{p(q-r)},p} \subset \dot{W}_{p,V}^{1}
$$  
 with $\|f\|_{\dot{W}_{p,V}^{1}}\leq C\|f\|_{\frac{q(p-r)}{p(q-r)},p}$, while  item 2. gives us as in section 5 for non-homogeneous Sobolev spaces, that 
$$\dot{W}_{p,V}^{1}\subset (\dot{W}_{r,V}^{1},\dot{W}_{q,V}^{1})_{\frac{q(p-r)}{p(q-r)},p}
$$ 
with $\|f\|_{\frac{q(p-r)}{p(q-r)},p}\leq C\|f\|_{\dot{W}_{p,V}^{1}}$. 
 \end{proof}
 Let $A_{V}=\left\lbrace q\in ]1,\infty]: V\in RH_{q}\right\rbrace$ and $q_{0}=\sup\limits A_{V}$, $B_{M}=\left\lbrace s\in[1,q_{0}[: (P_{s}) \textrm{ holds }\right\rbrace$ and $s_{0}=\inf B_{M}$.
\begin{cor}\label{CIH} For all $p,\,p_{1},\,p_{2}$ such that $ 1\leq  p_{1} <p<p_{2}<q_{0} $ with $p>s_0$, $\dot{W}_{p,V}^{1}$ is a real interpolation space between $\dot{W}_{p_{1},V}^{1}$ and $\dot{W}_{p_{2},V}^{1}$.
\end{cor}
 \section{Interpolation of Sobolev spaces on Lie Groups}
 Consider $G$ a connected Lie group. Assume that $G$ is unimodular and let $d\mu$ be a fixed Haar measure on $G$. Let $X_{1},...,X_{k}$ be a family of left invariant vector fields  such that the $X_{i}$'s satisfy a H\"{o}rmander condition. In this case the Carnot-Carath\'{e}odory metric $\rho$ is a true metric is a distance, and $G$ equipped with the distance  $\rho$ is complete and defines the same topology as  the topology of $G$ as manifold (see \cite{coulhon8} page 1148). 
 It is known that $G$ has an exponential growth or polynomial growth. In the first case, $G$ satisfies the local doubling property $(D_{loc})$ and admits a local Poincar\'{e} inequality $(P_{1loc})$. In the second case, it admits the global doubling property $(D)$ and a global Poincar\'{e} inequality $(P_{1})$ (see  \cite{coulhon8}, \cite{guivarch}, \cite{saloff2}, \cite{varopoulos2} for more details).

 \begin{dfn}[Sobolev spaces $W_{p,V}^{1}$]  For $1\leq p<\infty$ and for a weight $V\in A_{\infty}$, we define the Sobolev space $W_{p,V}^{1}$ as the completion of $C^{\infty}$ functions for the norm:
$$
  \| u \|_{W_{p,V}^{1}}=\|f \|_{p}+\|\,|Xf|\,\|_{p}+\|Vf\|_{p}
$$
where $|Xf|=\left(\sum_{i=1}^{k}|X_{i}f|^{2}\right)^{\frac{1}{2}}$.
\end{dfn}
\begin{dfn} We denote by $W_{\infty,V}^{1}$ the space of all bounded Lipschitz functions $f$ on $G$ such that $\|Vf\|_{\infty}<\infty$ which is a Banach space.
\end{dfn}
\begin{prop}  Let $V\in RH_{qloc}$ for some $1\leq q<\infty$. Consider, for $1\leq p<q$, 
$$H_{p,V}^{1}=\left\lbrace f\in L_{p}:\; |\nabla f|\, \textrm{ and } Vf \in L_{p} \right\rbrace
$$ and equip it with the same norm as $W_{p,V}^{1}$. 
Then as in Proposition \ref{DC} in the case of Riemannian manifolds, $C_{0}^{\infty}$ is dense in $H_{p,V}^{1}$ and hence $W_{p,V}^{1}=H_{p,V}^{1}$.
\end{prop}

%\begin{rem} We denote $H_{\infty}^{1}(G)=W_{\infty}^{1}(G)$ for the set of all bounded Lipschitz functions on $G$.
%\end{rem}
%\begin{prop} 
%\begin{itemize}
%\item[1.] $W_{p,V}^{1}(G)$ equipped with this norm is a Banach space.
%\item[2.] For $1\leq p<\infty$, $C_{0}^{\infty}(G)$ is dense in $W_{p}^{1}(G)$.
%\end{itemize}
%\end{prop}
\paragraph{\textbf{Interpolation of  $W_{p,V}^{1}$:}}
 Let $V \in RH_{qloc}$ for some $1<q\leq\infty$. To interpolate  the $W_{p_{i},V}^{1}$, we distinguish between the polynomial and the exponential growth cases. If $G$ has polynomial growth and $V\in RH_{q}$, then we are in the global case. Otherwise we are in the local case. In the two cases we obtain the following theorem:
\begin{thm}\label{G} Let $G$ be a connected Lie group as in the beginning of this section and assume that $V\in RH_{qloc}$ with $1<q\leq\infty$. Denote $T_{1}f=|f|+|X f|+|Vf|$, $T_{r*}f=|f|^{r*}+|X f|^{r*}+|Vf|^{r*}$ for $1\leq r<\infty$. %$T_{1**}f=|f|^{**}+|X f|^{**}+|V^{\frac{1}{2}}f|^{**}$
\begin{itemize}
\item[a.] If $q<\infty$, then
\begin{itemize}
\item[1.] there exists $C_{1}>0$ such that for every $f\in W_{1,V}^{1}+W_{q,V}^{1}$ and $t>0$ 
$$
K(f,t, W_{1,V}^{1},W_{q,V}^{1})\geq C_{1}\left\lbrace\left(\int_{0}^{t^{\frac{q}{q-1}}}T_{1*}f(u)du \right)^{\frac{1}{s}}+t\left(\int_{t^{\frac{q}{q-1}}}^{\infty}T_{q*}f(u)du\right)^{\frac{1}{q}}\right\rbrace ;
$$
\item[2.] for $1\leq p<\infty$, there exists $C_{2}>0$ such that for every $f\in W_{p,V}^{1}$ and $t>0$,
$$ K(f,t, W_{1,V}^{1},W_{q,V}^{1}) \leq C_{2}\left\lbrace \int_{0}^{t^{\frac{q}{q-1}}}T_{1*}f(u)du +t\left(\int_{t^{\frac{q}{q-1}}}^{\infty} (\mathcal{M}T_{1}f)^{*q}(u)du\right)^{\frac{1}{q}}\right\rbrace.
$$
\end{itemize}
\item[b.] If $q=\infty$, then for every $f \in
 W_{1,V}^{1}+W_{\infty,V}^{1}$ and $t>0$ 
\begin{equation*}
 K(f,t,W_{1,V}^{1},W_{\infty,V}^{1})\sim \int_{0}^{t}
T_{1*}f(u)du.
\end{equation*}
\end{itemize}
\end{thm}
 \begin{thm} Let $G$ be as above, $V\in RH_{qloc}$, for some $1<q\leq\infty$. Then, for $1\leq p_{1}< p < p_{2}<q_{0}$, $W^{1}_{p,V}$ is a real interpolation space between $W^{1}_{p_{1},V}$ and $W^{1}_{p_{2},V}$ where $q_{0}=\sup\left\lbrace q\in]1,\infty]:V\in RH_{qloc}\right\rbrace$. 
\end{thm}
\begin{proof}
Combine Theorem \ref{G} and the reiteration theorem.
\end{proof}
\begin{rem} For $V\in A_{\infty}$, define the homogeneous Sobolev spaces $\dot{W}_{p,V}^{1}$  as the vector space of distributions $f$ such that $Xf$ and $Vf \in L_{p}$ and equip this space with the norm 
$$
\|f\|_{\dot{W}_{p,V}^{1}}=  \|\,|Xf|\,\|_{p}+\|Vf\|_{p}
$$
and $\dot{W}_{\infty,V}^{1}$ as the space of all Lipschitz functions $f$ on $G$ with $\|Vf\|_{\infty}<\infty$. Theses spaces are Banach spaces. If $G$ has polynomial growth, we obtain interpolation results analog to those of section 6.
\end{rem}

\paragraph{\large{\textbf{Examples:}}}
For examples of spaces on which our interpolation result applies see section 11 of \cite{badr1}. 
\\
Examples of $RH_{q}$ weights in $\mathbb{R}^{n}$ for $q<\infty$ are the power weights $|x|^{-\alpha}$ with $-\infty<\alpha<\frac{n}{q}$ and positive polynomials for $q=\infty$. We give an other example of $RH_{q}$ weights on a Riemannian manifold $M$: consider $f,\,g\in L_{1}(M)$, $1\leq r<\infty$ and $1<s\leq\infty$, then $V(x)=\left(\mathcal{M}f(x)\right)^{-(r-1)}\in RH_{\infty}$ and $W(x)=\left(\mathcal{M}g(x)\right)^{\frac{1}{s}}\in RH_{q}$ for all  $q<s$ ( $q=s$ if $s=\infty$) and hence $V+W\in RH_{q}$ for all $q<s$ ( $q=s$ if $s=\infty$) (see \cite{auscher6}, \cite{auscher7} for details).

\section{Appendix}

\paragraph{\textbf{ Proof of Proposition \ref{DC}}:}
We follow the method of Davies \cite{davies1}. Let $L(f)=L_{0}(f)+L_{1}(f)+L_{2}(f):=\int_{M}|f|^{p}d\mu+\int_{M}|\nabla f|^{p}d\mu+\int_{M}|Vf|^{p}d\mu$.
We will prove the proposition in three steps:
\begin{itemize}
\item[1.] Let $f\in H_{p,V}^{1}$. Fix $p_{0}\in M$ and let $\varphi\in C_{0}^{\infty}(\mathbb{R})$ satisfies $\varphi\geq 0$, $\varphi(\alpha)=1$ if $\alpha<1$ and $\varphi(\alpha)=0$ if $\alpha>2$. Then put $f_{n}(x)=f(x) \varphi(\frac{d(x,p_{0})}{n})$. Elementary calculations establish that $f_{n}$ lies in $H_{p,V}^{1}$. Moreover,
\begin{align*}
L(f-f_{n})&=\int_{M}|f(x)\{1-\varphi(\frac{d(x,p_{0})}{n})\}|^{p}d\mu(x)
\\
&+\int_{M}|\nabla f(x)\{1-\varphi(\frac{d(x,p_{0})}{n})\}-n^{-1}f(x) \varphi'(\frac{d(x,p_{0})}{n})\nabla (d(x,p_{0}))|^{p}d\mu(x)
\\
&+\int_{M}|V^{\frac{1}{2}}(x)f(x)(1-\varphi(\frac{d(x,p_{0})}{n}))|^{p}d\mu(x) 
\\
&\leq \int_{M}|f(x)\{1-\varphi(\frac{d(x,p_{0})}{n})\}|^{p}d\mu(x)
\\
&+ 2^{p-1}\int_{M}|\nabla f(x)\{1-\varphi(\frac{d(x,p_{0})}{n})\}|^{p}d\mu(x)+2^{p-1}n^{-p}\int_{M}|f(x)|^{p}|\varphi'(\frac{d(x,p_{0})}{n})|^{p}d\mu(x)
\\
&+\int_{M}V^{p}(x)|f(x)|^{p}|1-\varphi(\frac{d(x,p_{0})}{n})|^{p}d\mu(x).
\end{align*}
This converges to zero as $n\rightarrow \infty$ by the dominated convergence theorem. Thus the the set of functions $f \in H_{p,V}^{1}$ with compact support is dense in $H_{p,V}^{1}$.
\item[2.] Let $f\in H_{p,V}^{1}$ with compact support. Let $n>0$ and $F_{n}: \mathbb{R}\rightarrow \mathbb{R}$ be a smooth increasing function such that
 \begin{equation*}
F_{n}(s)= \begin{cases}
s \quad\textrm{ if }\, -n\leq s\leq n,
\\ 
n+1 \quad\textrm{ if } s \geq n+2,
\\
-n-1 \quad\textrm{ if } s\leq -n-2
\end{cases}
\end{equation*}
and $0\leq F'_{n}(s)\leq 1$ for all $s\in \mathbb{R}$. If we put $f_{n}(x):=F_{n}(f(x))$ then $|f_{n}(x)|\leq |f(x)|$ and $\lim_{n\rightarrow \infty}f_{n}(x)=f(x)$ for all $x\in M$. The dominated convergence theorem yields
$$
\lim_{n\rightarrow\infty}L_{0}(f-f_{n})=\lim_{n\rightarrow \infty}\int_{M}|f-f_{n}|^{p}d\mu=0
$$
and 
$$
\lim_{n\rightarrow\infty}L_{2}(f-f_{n})=\lim_{n\rightarrow \infty}\int_{M}V^{p}|f-f_{n}|^{p}d\mu=0
$$
Also
\begin{align*}
\lim_{n\rightarrow \infty}L_{1}(f-f_{n})&=\lim_{n\rightarrow\infty}\int_{M}|\nabla f-F'_{n}(f(x))\nabla f|^{p}d\mu(x)
\\
&=\lim_{n\rightarrow\infty}\int_{M}|1-F'_{n}(f(x))|^{p}|\nabla f(x)|^{p}d\mu(x)
\\
&=0.
\end{align*}
Therefore the set of bounded functions $f\in H_{p,V}^{1}$ with compact support is dense in $H_{p,V}^{1}$.
\item[3.] Let now $f\in H_{p,V}^{1}$ be bounded and with compact support.
 Consider locally finite coverings of $M$, $(U_{k})_{k}$, $(V_{k})_{k}$ with $\overline{U_{k}}\subset V_{k}$, $V_{k}$ being endowed with a real coordinate chart $\psi_{k}$. Let $(\varphi_{k})_{k}$ be a partition of unity subordinated to the covering $(U_{k})_{k}$, that is, for all $k$, $\varphi_{k}$ is a $C^{\infty}$ function compactly supported in $U_{k}$, $0\leq\varphi_{k}\leq 1$ and $\sum_{k=1}^{\infty}\varphi_{k}=1$. There exists a finite subset $I$ of $\mathbb{N}$ such that $f=\sum_{k\in I}f\varphi_{k}:=\sum_{k\in I}f_{k}$. Take $\epsilon>0$. The functions $g_{k}=f_{k}\circ\psi_{k}^{-1}$ --which belongs to $W_{p}^{1}(\mathbb{R}^{n})$ since $f$ and $|\nabla f| \in L_{ploc}$-- can be approximated by smooth functions $w_{k}$ with compact support (standard approximation by convolution). The $w_k$ are defined as $w_{k}=g_{k}*\alpha_{k}$ where $\alpha_{k}\in C_{0}^{\infty}(\mathbb{R}^{n})$  is a standard mollifier, $\supp w_{k}\subset \psi_{k}(V_{k})$ and $\|g_{k}-w_{k}\|_{W_{p}^{1}} 
\leq \frac{\epsilon}{2^{k}}$. Define
\begin{equation*}
h_{k}(x)=\begin{cases}
w_{k}\circ \psi_{k}(x)\, \textrm{ if } x\in V_{k},
\\
0 \; \textrm{  otherwise}.
\end{cases}
\end{equation*}
Thus $\supp h_{k}\subset V_{k}$ and 
\begin{equation*}
\|f_{k}-h_{k}\|_{p}=\left(\int_{V_{k}}|f_{k}-h_{k}|^{p}d\mu\right)^{\frac{1}{p}}=\|g_{k}-w_{k}\|_{p} 
\leq \frac{\epsilon}{2^{k}}.
\end{equation*}
\begin{equation*}
\|\,|\nabla(f_{k}-h_{k})|\|_{p}=\left(\int_{V_{k}}|\nabla(f_{k}-h_{k})|^{p}d\mu\right)^{\frac{1}{p}}=\|\,|\nabla(g_{k}-w_{k})|\,\|_{p} 
\leq \frac{\epsilon}{2^{k}}.
\end{equation*}
Hence the series $\sum_{k\in I} (f_{k}-h_{k})$ is convergent in $W_{p}^{1}$. Moreover $\sum_{k\in I} (f_{k}-h_{k})=f-h_{\epsilon}$  where  $h_{\epsilon}=\sum_{k\in I}h_{k}$, and
$\|f-h_{\epsilon}\|_{W_{p}^{1}}\leq \sum_{k\in I} \|f_{k}-h_{k}\|_{W_{p}^{1}}\leq \epsilon$.\\
If $l_{\epsilon}:=|f-h_{\epsilon}|^{p}$ then $\lim_{\epsilon\rightarrow 0}\|l_{\epsilon}\|_{1}=0$ and there exists a compact set $K$ which contains the support of every $l_{\epsilon}$. We have $\|h_{\epsilon}\|_{\infty}\leq \sharp I \|f\|_{\infty}$ for all $\epsilon>0$. Indeed
\begin{align*}
\sum_{k\in I}|h_{k}(x)|&=\sum_{k\in I}\int_{\mathbb{R}^{n}}|g_{k}(y)|\,\alpha_{k}(\psi_{k}(x)-y)dy
\\
&=\int_{\mathbb{R}^{n}}\sum_{k\in I}|f\varphi_{k}(\psi_{k}^{-1}(y))|\,\alpha_{k}(\psi_{k}(x)-y)dy
\\
&\leq \|f\|_{\infty}\int_{\mathbb{R}^{n}}\sum_{k\in I}\varphi_{k}(\psi_{k}^{-1}(y))\,\alpha_{k}(\psi_{k}(x)-y)dy
\\
&\leq \|f\|_{\infty}\sum_{k\in I}\int_{\psi_{k}(U_{k})}\varphi_{k}(\psi_{k}^{-1}(y))\,\alpha_{k}(\psi_{k}(x)-y)dy
\\
&\leq \|f\|_{\infty}\sum_{k\in I}\int_{\mathbb{R}^{n}}\alpha_{k}(z)dz
\\
&\leq \sharp I \|f\|_{\infty}.
\end{align*}
It follows that $\|l_{\epsilon}\|_{\infty}\leq 2^{p-1}(1+\sharp I)\|f\|_{\infty}^{p}=C\|f\|_{\infty}^{p}$ ($C$ being independent of $\epsilon$ it depends just on $f$) for all $\epsilon>0$. We claim that these facts suffice to deduce that $\lim_{\epsilon\rightarrow 0}\int_{M}l_{\epsilon}V^{p}d\mu=0$, that is 
$$
\lim_{\epsilon\rightarrow 0}L_{2}(f-l_{\epsilon})=0.
$$
Hence $C_{0}^{\infty}$ is dense in $H_{p,V}^{1}$.
\item[4.] It remains to prove the above claim. Since $V\in RH_{ploc}$, there exists $r>p$ such that $V\in RH_{rloc}$ and therefore $V^{p}\in L_{t,loc}$ where $t=\frac{r}{p}>1$. Hence, by H\"{o}lder inequality we get
% $p<q$, there  exists a decomposition $V/K=V_{p}+V_{\infty}$ such that $\|V_{p}\|_{\frac{p}{2}}<\epsilon^{\frac{2}{p}}$ and $V_{\infty}\in L_{\infty}$. We then have
\begin{align*}
0\leq \int_{M}l_{\epsilon}V^{p}d\mu&=\int_{K}l_{\epsilon}V^{p}d\mu
\\
&\leq \|l_{\epsilon}\|_{L_{t'}(K)}\,\|V^{p}\|_{L_{t}(K)}
\\
&\leq C \|f\|_{\infty}^{\frac{p}{r}}\epsilon^{\frac{1}{t'}}
%\|l_{\epsilon}\|_{\infty}\|V_{p}^{\frac{p}{2}}\|_{1}+\|l_{\epsilon}\|_{1}\|V_{\infty}^{\frac{p}{2}}\|_{\infty}
%\\
%&\leq C\epsilon \|f\|_{\infty}^{p}+\|l_{\epsilon}\|_{1}\|V_{\infty}^{\frac{p}{2}}\|_{\infty}.
\end{align*}
 for all $\epsilon>0$, $t'$ being the conjugate exponent of $t$. The proof of Proposition \ref{DC} is therefore complete.
\end{itemize}
\bibliographystyle{plain}
\bibliography{latex2}
\end{document}